\documentclass[intlimits]{amsart}

\usepackage{amsmath,amsthm,amssymb,amsfonts}
\usepackage{color}
\usepackage{hyperref}
\usepackage{tikz}
\usepackage{float}
\usepackage{enumerate}

\newtheorem{theo}{Theorem}
\newtheorem{corollary}[theo]
{Corollary}
\newtheorem{lem}[theo]
{Lemma}
\newtheorem{proposition}[theo]
{Proposition}
\theoremstyle{definition}
\newtheorem{defi}[theo]
{Definition}
\newtheorem{remark}[theo]
{Remark}

\newcommand\R{{\mathbb{R}}}

\newcommand\ptl{{\partial}}
\newcommand\Char{{1\!\!1}}
\newcommand\sgn{{\rm sgn}}
\newcommand\conv{{\rm conv}}
\newcommand\ra{\rightarrow}
\newcommand\da{\downarrow}

\newcommand\mt{\mapsto}

\DeclareMathOperator*{\esssup}{ess \, sup}
\DeclareMathOperator*{\essinf}{ess \, inf}

\title[$L^1$ kinetic theory for fractional conservation laws]{Nonlocal dissipation measure and $L^1$ kinetic theory for fractional conservation laws}
\author[N. Alibaud]{Natha\"{e}l Alibaud}
\address[Natha\"{e}l Alibaud]{\'Ecole Nationale Sup\'erieure de M\'ecanique et des Microtechniques,
26 chemin de l'\'Epitaphe\\
and\\
Laboratoire de Math\'ematiques de Besan\c{c}on
UMR CNRS 6623,\\
Universit\'e de Bourgogne Franche-Comt\'e\\
16 route de Gray\\
25 030 Besan\c{c}on cedex,
France}
\email{nathael.alibaud\@@{}ens2m.fr}

\author[B. Andreianov]{Boris Andreianov}
\address[Boris Andreianov]{Institut Denis Poisson CNRS UMR7013\\
Universit\'e de Tours, Universit\'e d'Orl\'eans\\
Parc de Grandmont, 37200 Tours, FRANCE}
\email{boris.andreianov@univ-tours.fr}

\author[A. Ou\'edraogo]{Adama Ou\'edraogo}
\address[Adama Ou\'edraogo]{D\'epartement de Math\'ematiques,
Universit\'e Nazi Boni,
01 BP 1091 Bobo-Dioulasso, Burkina-Faso}
\email{adam\_ouedraogo3@yahoo.fr}

\subjclass[2010]{primary 35R11, 35L65; secondary 35K59, 35D99}

\keywords{Kinetic formulation, fractional conservation law, nonlocal and nonlinear diffusion, pure jump L\'evy operator,  
nonlinearity of porous medium kind,  dissipation measure, well-posedness, $L^1$ data}

\thanks{This research was supported by the ``French ANR
  project CoToCoLa, no. ANR-11-JS01-006-01.''  The second author thanks J.-Ph. Anker for a helpful discussion.}
 
\begin{document}

\thispagestyle{plain}

 \begin{abstract}
 We introduce a kinetic formulation for scalar conservation laws 
with nonlocal and nonlinear diffusion terms. We deal with merely $L^1$ initial data, general self-adjoint pure jump L\'evy operators, and locally Lipschitz nonlinearities of porous medium kind possibly strongly degenerate. The cornerstone of the formulation and the uniqueness proof is an adequate explicit representation of the nonlocal dissipation measure. This approach is inspired from the second order theory unlike the cutting technique previously introduced for bounded entropy solutions. The latter technique no longer seems to fit our setting. This is moreover the first time that the more standard and sharper tools of the second order theory are faithfully adapted to fractional conservation laws.
\end{abstract}

\maketitle

\tableofcontents

\newpage

\section{Introduction}

 In this paper, we define a kinetic formulation for scalar conservation laws with nonlocal and nonlinear diffusion terms. We consider initial-value problems of the form 
\begin{equation}\label{cauchy-problem}
\begin{cases}
\partial_{t}u+\nabla \cdot F(u)+g[A(u)]=0, & \quad (t,x) \in \R^+ \times \R^d,\\[1ex]
u(t=0,x)= u_0(x), & \quad x \in \R^d,
\end{cases}
\end{equation}
where $u=u(t,x)$ is the unknown function, $\nabla$ denotes the gradient operator with respect to $x$ and $g$ is an integro-differential operator properly defined at least on~$\mathcal{D}(\R^d)$ by
\begin{equation}\label{def-g}
g[\varphi](x):=- P.V. \int_{\R^d} (\varphi(x+z)-\varphi(x)) \mu(z) \, dz,
\end{equation}
where $\mu$ is a (Borel) measure and $\mu(z) \, dz$ abusively stands for $d \mu(z)$ or $\mu(dz)$. Throughout, the initial data $u_0$ is assumed merely integrable, and the other data are assumed to satisfy the following conditions:
\begin{align}
& F \in W^{1,\infty}_{\rm loc}(\R,\R^d){,}\label{condition-on-F}\\
& A \in W^{1,\infty}_{\rm loc}(\R) \text{ and is nondecreasing}{,}\label{condition-on-A}\\
& \mbox{$\mu \geq 0$ with $\mu(\{0\})=0$ and $\int_{\R^d} (|z|^2 \wedge 1) \mu(z) \, dz<\infty${,}}\label{condition-on-mu}\\
& \mbox{$\mu$ is even, i.e., it is invariant by the application $z \mt -z$.}\label{condition-symmetry}
\end{align}
 The principal value in \eqref{def-g} is defined as the limit
\begin{equation}\label{def-PV}
g[\varphi](x):=-\lim_{r \da 0} \int_{|z|>r} \left(\varphi(x+z)-\varphi(x)\right) \mu(z) \, dz,
\end{equation}
which makes sense because of \eqref{condition-on-mu}--\eqref{condition-symmetry}. For all fixed $r>0$, we also  have  
\begin{equation}\label{Levy-form}
g[\varphi](x)=-\int_{\R^d} \left(\varphi(x+z)-\varphi(x)-z \cdot \nabla \varphi(x) \mathbf{1}_{|z| \leq r}\right) \mu(z) \, dz.
\end{equation}

 The operators of the form \eqref{Levy-form} correspond to generators of pure jump L\'evy processes. 
By \eqref{condition-symmetry} we restrict to self-adjoint operators. 
They  constitute a general class of nonlocal diffusive operators  \cite{Sat99}. Their use in scalar conservation laws goes back to \cite{Ros89} on the Chapman-Enskog expansion involving a convolution operator corresponding to $\int \mu<\infty$. Scalar conservation laws with singular nonlocal diffusions were considered later  for instance in semiconductor growth~\cite{Woy01} 
or gas
detonations~\cite{Cla02}. 
A typical example of such a diffusion is the fractional Laplacian $(-\Delta)^\frac{\alpha}{2}$, $\alpha \in (0,2)$, corresponding to $\mu(z)=c\, |z|^{-d-\alpha}$; see \cite{Lan72}. 
More recent models led to further study scalar conservation laws with nonlocal and nonlinear diffusions; see for instance  \cite{RoYo07} on radiation hydrodynamics. Our setting covers all these problems.
As a byproduct, it includes the nonlocal diffusion equation $\partial_tu+g[A(u)]=0$ and allows for porous medium nonlinearities $A(u)=|u|^{m-1} u$ with $m \geq 1$; see \cite{DeQuRoVa11,DeQuRoVa12,DTEnJa17b,DTEnJa17a} for an extensive account on these PDEs. 
For other related PDEs with nonlocal and nonlinear diffusions, 
 see
 \cite{BiKaIm10,BiKaMo10,CaVa10} 
and the references therein.

\subsubsection*{State-of-the-art}


Equation~\eqref{cauchy-problem} may degenerate and should share properties with the scalar hyperbolic conservation law
\begin{equation}\label{scl}
\ptl_t u+\nabla \cdot F(u)=0
\end{equation}
and the degenerate parabolic equation
\begin{equation}\label{dpe}
\ptl_t u+\nabla \cdot F(u)-\triangle A(u)=0.
\end{equation}
Their main difficulties 
are the possible creation of singularities and the nonuniqueness of weak solutions. Since the fundamental work of Kruzhkov
\cite{Kru70} who defined entropy solutions
 for first order equations  
and established well-posedness in the $L^\infty$ framework, many  other well-posedness results 
were obtained.
We refer to
Carrillo \cite{Car99} for the entropy formulation  of  elliptic-parabolic-hyperbolic  problems involving 
Leray-Lions type operators.
The more delicate anisotropic diffusion case has been treated
by Bendahmane and Karlsen
in~\cite{BeKa04} using, in particular, the insight from the paper~\cite{ChPe03} of Chen and Perthame. The setting of \cite{ChPe03} is different because the kinetic formulation is used to achieve well-posedness;
this concept goes back to \cite{Bre83,GiMi83} and to the classical work \cite{LiPeTa91,LiPeTa94} of Lions, Perthame and Tadmor. An extensive account on the kinetic formulation of conservation laws can be found in \cite{Per09}.
One of the advantages of the kinetic formulation is that the $L^1$ space  is natural  for both existence and  uniqueness. 

{S}everal authors have
extended the notion of  $L^\infty$  entropy solutions to nonlocal problems 
of the form \eqref{cauchy-problem}. We refer  to \cite{KaNi98,KaNi99,LiTa01,LaMa03,Ser03,RoYo07} and the references therein for  the case $\int \mu<\infty$.  The case of singular 
 operators  is more delicate  and the first results   were concerned with conservation laws with memory
 nonlocal in time. 
The adequate notion of entropy solutions  was introduced in \cite{CoGrLo96} by Cockburn, Gripenberg and Londen; see also \cite{JaWi03}.
 It was adapted later in~\cite{Ali07} for fractional diffusions in space with a focus on the equation
\begin{equation}\label{fcl}
\ptl_t u+\nabla \cdot F(u)+(-\triangle)^\frac{\alpha}{2} u=0.
\end{equation}
The pioneering work on  \eqref{fcl}  goes back to Biler, Funaki and Woyczy\'{n}ski \cite{BiFuWo98}. Now the well-posedness 
 is  well-understood: If  $\alpha \geq 1$,  there is a unique smooth solution \cite{BiFuWo98,DrGaVo03,ChCz10,Sil11,CoVi12};  if $\alpha<1$, shocks can
occur~\cite{AlDrVo07,KiNaSh08} and weak solutions can be non-unique \cite{AlAn10}; for any $\alpha \in (0,2)$,
there exists a unique entropy solution corresponding to the classical one when it exists as well \cite{Ali07}. The entropy solution theory
was finally  extended by Karlsen and Ulusoy  \cite{KaUl11} to pure  jump L\'evy operators  and by Cifani and Jakobsen  \cite{CiJa11} to nonlinear diffusions such as in \eqref{cauchy-problem}. 

 As concerning $L^1$ data,
Wei, Duan and Lv recently introduced a kinetic formulation
in \cite{WeiDuanLv} for $L^1 \cap BV$ solutions  of \eqref{fcl} when $\alpha <1$. This situation is very particular since 
$(-\triangle)^{\frac{\alpha}{2}} u
\in L^1$ and this allows to treat the diffusion as a zero order term. In general, we do not know whether $A(u)$ is locally integable for unbounded solutions and non-globally Lipschitz $A$. We thus have difficulties to define $g[A(u)]$ as a distribution. 

To conclude this state-of-the-art, note that the concept of renormalized solutions \cite{BeCaWi00,BeKa05} provides another framework for $L^1$ data. It was extended to nonlocal problems but essentially for elliptic PDEs as far as we know; see e.g. \cite{AlAnBe10,KlimsiakRozkosz}.

\subsubsection*{Main contribution}

The present paper extends the kinetic formulation of Lions, Pethame and Tadmor \cite{LiPeTa91,LiPeTa94} to \eqref{cauchy-problem}; see Definition \ref{defi-kinetic}. The main results are the equivalence with entropy solutions in $L^1 \cap L^\infty$, the well-posedness in $L^1$ and the $L^1$ contraction principle; see Theorems \ref{thm:equivalence} and \ref{thm:wp-kinetic}. The cornerstone of the theory is a new explicit representation of the nonlocal dissipation measure in the spirit of Chen and Perthame \cite{ChPe03}; see  Theorem \ref{key-theo}. 

Technically, the uniqueness proof does not rely on the {\it cutting technique} used in every known proofs on $L^\infty$ entropy solutions. 
This technique was based on first order like arguments which no longer seem to fit our setting. Here we faithfully adapt the more standard and sharper tools of the second order theory and this is our main technical contribution. To give more details, we need to recall some facts on entropy solutions.

\subsubsection*{Various entropy inequalities}

An entropy solution  to \eqref{scl}  is a function $u=u(t,x)$ such that for all $\xi \in \R$,
\begin{equation*}
\partial_t |u-\xi|+\nabla_x \cdot \left\{\sgn (u-\xi) (F(u)-F(\xi))\right\} \leq 0
\end{equation*}
in $\mathcal{D}'((0,\infty) \times \R^d)$, where $\sgn (\cdot):=|\cdot|'$ is the sign function.
 The uniqueness can be achieved   
by using the Kruzhkov device of doubling the variables \cite{Kru70}. 
For second order equations, one needs to take into account some form of parabolic dissipation.   
This can  be achieved in two ways. The first way was developed by Carrillo in \cite{Car99} and consists in recovering 
 such a dissipation 
from the entropy inequalities. The second way was introduced in \cite{ChDi01} by Chen and DiBenedetto and consists in explicitly including a proper form of 
 dissipation  in the entropy inequalities; see also~\cite{KaRi03,ChPe03,BeKa05,ChKa05}. For  
\eqref{dpe}, 
this gives  
\begin{equation}
\label{ei-dpe}
\begin{split}
& \partial_t |u-\xi|+\nabla_x \cdot \left\{\sgn (u-\xi) (F(u)-F(\xi))\right\}\\
& -\triangle_x |A(u)-A(\xi)|
\leq -2 \delta(u-\xi) A'(u) |\nabla u|^2,
\end{split}
\end{equation}
where $\delta(\cdot):=\frac{1}{2} \sgn'(\cdot)$ is the Dirac measure at zero. 
%
In \cite{CiJa11}, Cifani and Jakobsen 
 used the following entropy inequalities for 
\eqref{cauchy-problem}:  
\begin{equation}\label{ei-fdpe}
\begin{split}
& \partial_t |u-\xi|+\nabla_x \cdot \left\{\sgn (u-\xi) (F(u)-F(\xi)) \right\}\\
& - P.V. \int_{|z|\leq r} \left(|A(\tau_z u)-A(\xi)| -|A(u)-A(\xi)|\right) \mu(z) \, dz\\
& - \sgn (u-\xi) \int_{|z|> r} (A(\tau_z u)-A(u)) \mu(z) \, dz \leq 0,
\end{split}
\end{equation}
for all $\xi \in \R$ and all $r>0$, where $\tau_z u$ designs the function $(t,x) \mt u(t,x+z)$. 
 The idea is to treat the integral in $|z| > r$ as a zero order term and neglect the other integral as~$r \da 0$; cf. also \cite{CoGrLo96,JaWi03,Ali07};  

 A natural question is whether it is possible to reformulate \eqref{ei-fdpe} in the spirit of \eqref{ei-dpe}. A first try was attempted by Karlsen and Ulusoy in \cite{KaUl11}  
but their uniqueness proof reduces to recover \eqref{ei-fdpe}.  
 Here we reformulate \eqref{ei-fdpe} with a new proper form of nonlocal dissipation: 
\begin{equation}\label{nei-fdpe}
\begin{split}
& \partial_t |u-\xi|+\nabla_x \cdot \left\{\sgn (u-\xi) (F(u)-F(\xi)) \right\}\\
& +g_x[|A(u)-A(\xi)|] \leq -\underbrace{\int_{\R^d} |A(\tau_z u)-A(\xi)| \Char_{\conv \{u,\tau_z u\}}(\xi) \mu(z) \, dz}_{=:n},
\end{split}
\end{equation}
where $\Char_{\conv \{u,\tau_z u\}}(\cdot)$ is the characteristic function of the real  interval of extremities~$u$ and $\tau_z u$.
 This will be the cornerstone of the kinetic theory. 

\subsubsection*{A successfull kinetic formulation}

Following \cite{LiPeTa91,LiPeTa94,ChPe03,Per09}, we obtain the kinetic equation in $\chi(\xi;u):=\partial_\xi \left(|\xi|-|\xi-u|\right)/2$ by derivating  \eqref{nei-fdpe} in $\xi$: 
\begin{equation}
\label{kf}
\partial_t \chi(\xi;u)+F'(\xi) \cdot \nabla_x \chi(\xi;u)+A'(\xi) g_x[\chi(\xi;u)]=\partial_\xi (m+n)
\end{equation}
in $\mathcal{D}'((0,\infty) \times \R^{d+1})$ where  
 $m \geq 0$ is unknown and $n \geq 0$  is 
 defined as in~\eqref{nei-fdpe}.
The advantage of \eqref{kf} when $u$ is merely $L^1$ is that it makes sense even if $F(u)$ and $A(u)$ are not $L^1_{\rm loc}$. 
It seems better to consider the kinetic equation from \eqref{nei-fdpe}, because derivating \eqref{ei-fdpe} would give the term
$
\delta(u-\xi) \int_{|z|> r} (A(\tau_z u)-A(\xi)) \mu(z) \, dz
$
which does not obviously make sense. Small and large jumps $z$ are thus treated in the same way in our uniqueness proof thanks roughly speaking to the explicit dissipation in \eqref{nei-fdpe} and Lemma~\ref{Lemma-F(a,b,c,d)-leq_G(a,b,c,d)}.

\subsubsection*{Outline of the paper} 
Our main 
results are stated in Section \ref{sec:def-results} and proved 
in Sections~\ref{sec:prelim}--\ref{sec:uniq}. The uniqueness 
 is first proved formally and then rigorously in the spirit of \cite{ChPe03,Per09}. 
It is the core of the paper. The existence could be established without relying on entropy solutions, making the kinetic theory self-sufficient as in \cite{ChPe03,Per09}, but here we use a known existence result for entropy solutions to be brief.

\subsubsection*{Reminders of our notation} The symbol $\nabla$ is used for the gradient in $x$  and $\nabla^2$ for the Hessian. The symbol $\delta(\xi)$ designs the Dirac mass at $\xi=0$.  In integrals, we use  the notation $\delta(\xi) \, d\xi$ for $d\delta(\xi)$ or $\delta(d\xi)$. We  do the same for  the other measures $\mu(z)$, $m(t,x,\xi)$, etc. 
 Note that $\mu$  in \eqref{condition-on-mu} is $\sigma$-finite  and Fubini's theorem applies  to define its product with  $dx$, etc., which we denote by $\mu(z) \, dx \, dz$, etc. 
Further notation is introduced in Section~\ref{sec:prelim}. 

\section{Entropy and kinetic solutions: Definitions and main results}\label{sec:def-results}

 Let us now give the rigorous definition of entropy and kinetic solutions and state our main results.
In order to avoid unnecessary technical issues, we only use $C^2$ entropies $u \mt S(u)$. 

Let us recall the notion of entropy solutions  to  \eqref{cauchy-problem} from \cite{CiJa11}, see also~\cite{CoGrLo96,JaWi03,Ali07,KaUl11}.
 Given any convex $C^2$ function $S:\R \ra \R$,  we consider 
$$
\eta:\R \ra \R^d \quad \mbox{and} \quad \beta:\R \ra \R,
$$satisfying
$$
\eta'=S' F' \quad \mbox{and} \quad \beta'=S' A'.
$$Throughout such a triplet $(S,\eta,\beta)$ is refered to as an {\em entropy-entropy flux triple.}

\begin{defi}[Entropy solutions]\label{def:entropy}
Let $u_0 \in L^1 \cap L^\infty(\R^d)$ and \eqref{condition-on-F}--\eqref{condition-symmetry} hold. A function $u\in L^\infty (\R^+;L^1(\R^d)) \cap L^\infty(\R^+ \times \R^d)$
is an entropy solution of~\eqref{cauchy-problem} provided that for all entropy-entropy flux triple $(S,\eta,\beta)$, all $r>0$, and all nonnegative test function $\varphi \in \mathcal{D}(\R^{d+1})$,
\begin{equation}\label{rei-fdpe}
\begin{split}
& \int_{0}^{\infty} \int_{\R^d} \left(S(u)\partial_t\varphi+\eta(u)\cdot \nabla \varphi\right) dt \, dx \\
& +\int_{0}^{\infty} \int_{\R^d} \int_{|z| > r}  S'(u(t,x)) \left(A(u(t,x+z))-A(u(t,x))\right) \varphi(t,x) \mu(z)  \, dt \, dx \, dz \\
& +P.V. \int_{0}^{\infty} \int_{\R^d} \int_{|z| \leq r} \beta(u(t,x)) \left(\varphi(t,x+z)-\varphi(t,x)\right) \mu(z) \, dt \, dx \, dz\\
& +\int_{\R^d}
S(u_0(x)) \varphi(0,x) \, dx\geq 0.
\end{split}
\end{equation}
\end{defi}

\begin{remark}
\begin{enumerate}
\item The original definition of \cite{CiJa11} was actually given with the entropies of Kruzhkov. The definition above is an equivalent reformulation already used for instance in \cite[Section 6]{KaUl11} or \cite[Section 7]{AlCiJapr}.
\item The principal value makes sense by \eqref{condition-on-mu}--\eqref{condition-symmetry} (see \eqref{def-PV} and \eqref{Levy-form}).
\item Here it may look that we are integrating a Lebesgue measurable function $u$ with respect to the Borel measure $\mu$. An easy way to avoid such measurability issues consists in only considering Borel representative of $u$.
\item As usually, classical solutions are entropy solutions and entropy solutions are weak (distributional) solutions; see \cite{CiJa11} for more details.
\end{enumerate}
\end{remark}

Here is the well-posedness result from~\cite{CiJa11}.
\begin{theo}[Well-posedness of entropy solutions]\label{thm:wp-entropy}
Let~$u_0 \in L^1\cap L^\infty(\R^d)$ and let us assume \eqref{condition-on-F}--\eqref{condition-symmetry}.
Then there
exists a unique entropy solution $u$ of \eqref{cauchy-problem}. This solution belongs to $C([0,\infty);L^1(\R^d)) \cap L^\infty(\R^+ \times \R^d)$ with $u(0,\cdot)=u_0(\cdot)$. Moreover, we have $\essinf u_0 \leq u \leq \esssup u_0$,
\begin{equation*}
\|u(t,\cdot)\|_{L^1(\R^d)}\leq\|u_0\|_{L^1(\R^d)} \quad \mbox{for all $t \geq 0$,}
\end{equation*}
and if $\tilde u$ is the solution of~\eqref{cauchy-problem}
associated to $\tilde u(0,\cdot)=\tilde u_0(\cdot) \in L^1 \cap L^\infty(\R^d)$, then
\begin{equation*}
\|u(t,\cdot)-\tilde u(t,\cdot)\|_{L^1(\R^d)} \leq \|u_0-\tilde u_0\|_{L^1(\R^d)}  \quad \mbox{for all $t \geq 0$.}
\end{equation*}
\end{theo}

 To  motivate our kinetic formulation, we need to reformulate the entropy inequalities with a proper form of nonlocal dissipation in the spirit of \cite{ChDi01}, see also \cite{KaRi03,ChPe03,BeKa05,ChKa05,KaUl11}. The key point is the following elementary Taylor's identity: For all $a,b \in \R$,
\begin{equation}\label{key-taylor}
S'(a) (A(b)-A(a))=\beta(b)-\beta(a)- \int_{\R} S''(\xi) |A(b)-A(\xi)| \Char_{\conv \{a,b\}}(\xi) \, d\xi,
\end{equation}
where throughout the paper, $\conv \{a,b\}$ stands for the interval
\begin{equation*}
\conv \{a,b\}:= \left(\min \{a,b\},\max \{a,b\}\right),
\end{equation*}
and $\Char_{\conv \{a,b\}}(\cdot)$ denotes its characteristic function normalized (everywhere defined) by
\begin{equation}
\label{def-char}
\Char_{\conv \{a,b\}}(\xi):=
\begin{cases}
1 & \mbox{if $\min \{a,b\} < \xi < \max \{a,b\}$},\\
\frac{1}{2} & \mbox{if $\xi =a$ or $b$,}\\
0 & \mbox{otherwise.}
\end{cases}
\end{equation}

\begin{remark}
In \eqref{def-char}, the choice of the value $\frac{1}{2}$ at the endpoints of the interval $\conv\{a,b\}$
 is dictated by the regularization procedure exploited in the uniqueness proof, since nonlinearities $S$ with singular second derivative are used.
 In this section, this technical detail can be neglected since $S$ is assumed to have the $C^2$ regularity.
\end{remark}

\begin{remark}
In the sequel, we will use other characteristic functions defined as usually. To avoid confusion, they will be denoted by $\mathbf{1}$ (and not $\Char)$. For instance, if $E \subseteq \R$, then
\begin{equation*}
\mathbf{1}_{E}(\xi):=
\begin{cases}
1 & \mbox{if } \xi \in E,\\
0 & \mbox{if not.}
\end{cases}
\end{equation*}
\end{remark}

The result below is a simple rewritting of the entropy inequality \eqref{rei-fdpe} based on the identity \eqref{key-taylor} and a passage to the limit as $r \da 0$.

\begin{theo}[Explicit representation of the nonlocal dissipation]\label{key-theo}
Assume \eqref{condition-on-F}--\eqref{condition-symmetry} and let $u_0 \in L^1 \cap L^\infty(\R^d)$. A function $u\in L^\infty \left(\R^+;L^1(\R^d)\right) \cap L^\infty(\R^+ \times \R^d)$
is an entropy solution of~\eqref{cauchy-problem} if and only if for all entropy-entropy flux triple $(S,\eta,\beta)$ and all nonnegative $\varphi \in \mathcal{D}(\R^{d+1})$,
\begin{equation}
\label{rnei-fdpe}
\begin{split}
& \int_{0}^{\infty} \int_{\R^d} \left(S(u) \partial_t\varphi+\eta(u) \cdot \nabla \varphi -\beta(u) g[\varphi]\right) dt \, dx \\
& +\int_{\R^d}
S(u_0(x))\varphi(0,x) \, dx \geq \int_{0}^{\infty} \int_{\R^{d+1}} S''(\xi) n(t,x,\xi) \varphi(t,x) \, dt \, dx \, d\xi,
\end{split}
\end{equation}
where the function $n:\R^+ \times \R^{d+1} \ra [0,\infty]$ is defined by
\begin{equation}
\label{def-nonloc-dissip}
n(t,x,\xi) := \int_{\R^d} |A(u(t,x+z))-A(\xi)| \Char_{\conv \{u(t,x),u(t,x+z)\}}(\xi) \mu(z) \, dz.
\end{equation}
\end{theo}
The proof of Theorem~\ref{key-theo} is deferred to   Section~\ref{sec:equiv}.  

\begin{remark}\label{absolutely-continuous}
Throughout $n$ is refered to as the {\em nonlocal dissipation measure.} This is a nonlocal version of the parabolic dissipation measure $2\delta(u-\xi) A'(u) |\nabla u|^2$ obtained for the degenerate parabolic equation \eqref{dpe}. Here, we get a measure absolutely continuous with respect to the Lebesgue one, since
$n \in L^1(\R^+ \times \R^{d+1})$ as a consequence of Inequality \eqref{rnei-fdpe} with $S'' \equiv 1$.
\end{remark}

We are now ready to define the notion of kinetic solutions. We use the framework of \cite{LiPeTa91,LiPeTa94,Per09} and especially the insight of \cite{ChPe03} by including the identification of the dissipation measure in the formulation. We consider the kinetic function $\chi:\R^2 \to \{-1,0,1\}$ defined by
\begin{equation}\label{def-chi}
\chi(\xi;u):=
\begin{cases}
1 & \mbox{if $0<\xi<u$},\\
-1 & \mbox{if $u<\xi<0$},\\
0 & \mbox{otherwise.}
\end{cases}
\end{equation}
Note that $u \in L^\infty(\R^+;L^1(\R^d))$ if and only if $\chi(\xi;u(t,x)) \in L^\infty(\R^+;L^1(\R^{d+1}))$. Note also that one has the following simple representation:
\begin{equation}\label{fundamental-theo}
S(u)-S(0)=\int_{\R} S'(\xi) \chi(\xi;u) \, d\xi.
\end{equation}
These observations lead us to the definition below where Inequality \eqref{rnei-fdpe} is roughly speaking rewritten with the $\chi$ function. Throughout the paper,
$$
\text{$M^1$ stands for the space of bounded Borel measures}
$$and
$$\text{$L^\infty_0$ stands for the space of 
a.e. bounded functions vanishing at infinity.}
$$
\begin{defi}[Kinetic solutions]\label{defi-kinetic}
Let $u_0 \in L^1(\R^d)$ and \eqref{condition-on-F}--\eqref{condition-symmetry} hold. A function $u\in L^\infty (\R^+;L^1(\R^d))$
is a kinetic solution of~\eqref{cauchy-problem} provided that there exists a nonnegative measure $m \in M^1_{\rm loc}([0,\infty) \times \R^{d+1})$ such that for almost all $\xi \in \R$,
\begin{equation}\label{cond-infinity}
\int_0^\infty \int_{\R^{d}} (m+n)(t,x,\xi) \, dt \, dx \leq \nu(\xi),
\end{equation}
for some $\nu \in L^\infty_0(\R_\xi)$, and for all~$\varphi \in \mathcal{D}(\R^{d+2})$,
\begin{equation}\label{rkf}
\begin{split}
& \int_{0}^{\infty} \int_{\R^{d+1}}\chi(\xi;u) \left(\partial_t\varphi+F'(\xi) \cdot \nabla_x \varphi-A'(\xi) g_x[\varphi]\right) dt \, dx \, d\xi \\
& +\int_{\R^{d+1}}
\chi(\xi;u_0(x))\varphi(0,x,\xi) \, dx \, d\xi = \int_{0}^{\infty} \int_{\R^{d+1}} (m+n)(t,x,\xi) \partial_\xi \varphi(t,x,\xi) \, dt \, dx \, d\xi,
\end{split}
\end{equation}
where $n$ is the nonnegative function defined in~\eqref{def-nonloc-dissip}.
\end{defi}

 Note that \eqref{cond-infinity} has to be understood in the distribution sense, that is to say
\begin{equation}\label{nu-weak-sense}
\int_0^\infty \int_{\R^{d+1}} (m+n)(t,x,\xi) \varphi(\xi) \, dt \, dx \, d\xi \leq \int_{\R} \nu(\xi) \varphi(\xi) \, d\xi,
\end{equation}
for any nonnegative $\varphi \in \mathcal{D}(\R_\xi)$, where hereafter, we {sometimes} write $\R_\xi$ to highlight dependence of functions and measures on the kinetic variable $\xi$. 

\begin{remark}
 The measure  $m$ is referred to as the {\em entropy defect measure.} It is a priori unknown but a posteriori uniquenely determined (see Theorem~\ref{thm:wp-kinetic}); in other words, the couple $(u,m)$ is the unknown of the kinetic formulation \eqref{rkf}.
\end{remark}

The next remark enumerates standard properties of kinetic solutions which remain valid in the nonlocal setting. Most of them will not be needed, so we refer to the arguments of \cite{LiPeTa91,LiPeTa94,ChPe03,Per09} for proofs.

\begin{remark}
\begin{enumerate}
\item If { $u_0 \in L^1 \cap L^\infty(\R^d)$,}
$$
{\rm supp} (m+n) \subseteq \left\{\essinf u_0\leq \xi \leq \esssup u_0\right\}.
$$
\item If { $u_0 \in L^1 \cap L^2(\R^d)$,}
\begin{align*}
\int_0^\infty \int_{\R^{d+1}} (m+n)(t,x,\xi) \, dt \, dx \, d\xi \leq \frac{1}{2} \|u_0\|_{L^2(\R^d)}^2.
\end{align*}
\item If $u_0$ is merely integrable, then for almost all $\xi \in \R$,
$$
\int_0^\infty \int_{\R^{d}} (m+n)(t,x,\xi) \, dt \, dx \leq \nu(\xi),
$$where $\nu(\xi):=\|(u_0-\xi)^+ \mathbf{1}_{\xi > 0}\|_{L^1(\R^d)}+\|(u_0-\xi)^- \mathbf{1}_{\xi < 0}\|_{L^1(\R^d)} \in L^\infty_0(\R_\xi)$.
\item If $S_A(u_0) \in L^1(\R^d)$ where
$
S_A(\xi):=\int_0^\xi \left(A(\zeta)-A(0)\right) d\zeta,
$
\begin{align*}
\int_0^\infty \int_{\R^{2d}} \left(A(u(t,x+z))-A(u(t,x))\right)^2 \mu(z) \, dt \, dx \, dz \leq \|S_A(u_0)\|_{L^1(\R^d)}.
\end{align*}
This should be compared to the usual $H^1$ estimate for degenerate parabolic equations \cite{Car99}, since for $g=(-\triangle)^{\frac{\alpha}{2}}$ the above estimate reads 
\begin{align*}
|A(u)|_{L^2 (\R^+;H^\frac{\alpha}{2}(\R^d))} \leq \sqrt{\|S_A(u_0)\|_{L^1(\R^d)}}.
\end{align*}
\end{enumerate}
\end{remark}

 Here is another standard property of convection-diffusion conservation laws that  remains valid for the nonlocal case, and which we  often  use throughout. It says that we can reformulate  Definition \ref{defi-kinetic}
  by expressing the initial data in the classical sense.

\begin{proposition}
\label{continuity-in-time}
 Assume that \eqref{condition-on-F}--\eqref{condition-symmetry} hold and that $u_0 \in L^1(\R^d)$. Then $u \in L^\infty(\R^+;L^1(\R^d))$ is a kinetic solution of \eqref{cauchy-problem} if and only if there exsits a nonnegative measure $m \in M_{\rm loc}^1([0,\infty) \times \R^{d+1})$ such that \eqref{cond-infinity} holds together with the following conditions:
\begin{equation}\label{kds}
\begin{cases}
\partial_t \chi(\xi;u)+F'(\xi) \cdot \nabla_x \chi(\xi;u)+A'(\xi) g_x[\chi(\xi;u)]=\partial_\xi (m+n),\\[1ex]
\lim_{t \da 0} \|u(t,\cdot)-u_0(\cdot)\|_{L^1(\R^d)}=0,\\[1ex]
\text{and } \lim_{t \da 0} \int_0^t \int_{\R^d} \int_{-R}^R (m+n)(s,x,\xi) \, ds \, dx \, d\xi=0, \quad \forall R>0,
\end{cases}
\end{equation}
where the equation holds in $\mathcal{D}'((0,\infty) \times \R^d \times \R_\xi)$ and the limits are taken in the essential sense.
\end{proposition}

We can now state the two main results of this paper.
\begin{theo}[Equivalence between entropy and kinetic solutions]\label{thm:equivalence}
Let \eqref{condition-on-F}--\eqref{condition-symmetry} hold, $u_0 \in L^1 \cap L^\infty(\R^d)$ and $u \in L^\infty(\R^+;L^1(\R^d)) \cap L^\infty(\R^+ \times \R^d)$. Then $u$ is an entropy solution of \eqref{cauchy-problem} if and only if it is a kinetic solution of \eqref{cauchy-problem}.
\end{theo}

\begin{theo}[Well-posedness in the pure $L^1$ setting]\label{thm:wp-kinetic}
Let~$u_0 \in L^1(\R^d)$ and let us assume \eqref{condition-on-F}--\eqref{condition-symmetry}.
Then there
exists a unique kinetic solution~$u$ of ~\eqref{cauchy-problem} and a unique measure $m$ satisfying Definition \ref{defi-kinetic}. This solution belongs to $C
([0,\infty);L^1(\R^d))$ with $u(0,\cdot)=u_0(\cdot)$. Moreover,
\begin{equation*}
\|u(t,\cdot)\|_{L^1(\R^d)}\leq\|u_0\|_{L^1(\R^d)} \quad \mbox{for all $t \geq 0$}
\end{equation*}
and if $\tilde u$ is the solution of~\eqref{cauchy-problem}
associated to ~$\tilde u(0,\cdot)=\tilde u_0(\cdot) \in  L^1(\R^d)$, then
\begin{equation*}
\|u(t,\cdot)-\tilde u(t,\cdot)\|_{L^1(\R^d)} \leq \|u_0-\tilde u_0\|_{L^1(\R^d)}  \quad \mbox{for all $t \geq 0$.}
\end{equation*}
\end{theo}

\begin{remark}
More generally, we have
\begin{equation*}
\|(u(t,\cdot)-\tilde{u}(t,\cdot))^\pm\|_{L^1(\R^d)}\leq \|(u_0-\tilde{u}_0)^\pm\|_{L^1(\R^d)},
\end{equation*}
so that $u_0\leq \tilde u_0$ entails $u\leq \tilde u$.
\end{remark}
The rest of this paper is devoted to the proofs of these results. The proof of Proposition \ref{continuity-in-time} is given in Appendix~\ref{appendix:prop}, for the sake of completeness. 

\section{Preliminary lemmas}\label{sec:prelim}

In this section, we give some basic results that will be useful in the sequel.

\subsection{Main properties of the kinetic function $\chi$} 
 
 Let us begin with some properties concerning the  function  defined in \eqref{def-chi}. We omit the proofs which can be found in \cite{Per09}, for instance. 
 
 Let us  first make precise the definition of the sign function used throughout:
\begin{equation*}
\sgn(\xi):=
\begin{cases}
1 & \mbox{if } \xi >0,\\
-1 & \mbox{if } \xi<0,\\
0 & \mbox{if } \xi=0.
\end{cases}
\end{equation*}

\begin{lem}\label{properties-of-khi}
\begin{enumerate}[{\rm (i)}]
\item For any reals $u$ and $\xi$,
$$
\sgn(\xi)\chi(\xi;u) =|\chi(\xi;u)|=\left(\chi(\xi;u)\right)^2.
$$
\item For any reals $u$ and $\tilde{u}$,
$$
 |u-v|
=\int_{\R} |\chi(\xi;u)-\chi(\xi;\tilde{u})| \, d \xi=\int_{\R} \left(\chi(\xi;u)-\chi(\xi;\tilde{u})\right)^{2} d\xi.
$$
\end{enumerate}
\end{lem}

\begin{remark}\label{isometry}
The map
$
u\in L^{\infty}(\R^+; L^{1}(\R^{d}))  \mapsto \chi(\xi;u) \in L^{\infty}(\R^+; L^{1}(\R^{d+1}))
$
is thus an isometry.
\end{remark}


\subsection{Main properties of the nonlocal diffusion operator $g$} Let us continue with standard results on  the operator defined in \eqref{def-PV}.  The proofs are gathered in Appendix \ref{appendix:tech} for the sake of completeness; see also \cite{Lan72,Sat99,CiJa11,AlCiJa12,AlCiJapr}. 

Let us first precise the sense of $g$ for sufficiently regular functions.
\begin{lem}\label{lem:well-def-g}
Under \eqref{condition-on-mu}--\eqref{condition-symmetry}, $g$ is  still  well defined by \eqref{def-PV} from
$W^{2,1} (\R^d)$
(resp. $C_b^2(\R^d)$)
into $L^1(\R^d)$ (resp. $C_b(\R^d)$). It is moreover linear, bounded, and
\begin{equation}\label{eq:IBPformula}
\int_{\R^d} \varphi g[f] \, dx=\int_{\R^d} f g[\varphi] \, dx \quad \forall f \in W^{2,1}(\R^d), \forall \varphi \in C^2_b(\R^d).
\end{equation}

\end{lem}

\begin{remark}
If $f$ is merely integrable, we can thus define $g[f]$ in the distribution sense by $\langle g[f],\varphi \rangle_{\mathcal{D}',\mathcal{D}}:=\int f g[\varphi]$.
\end{remark}


%
%

Let us continue with another useful formula; it is interpreted as an {integration by parts} formula involving the square root of $g$. 

\begin{lem}[Bilinear form]
\label{lem:bf}
Assume \eqref{condition-on-mu}--\eqref{condition-symmetry}, $f \in W^{1,1}(\R^d)$ and $\tilde{f} \in W^{1,\infty}(\R^d)$.
Then
\begin{equation*}
\int_{\R^d} \tilde{f} g[f] \, dx =\frac{1}{2} \int_{\R^{2d}}  (f(x)-f(y))(\tilde{f}(x)-\tilde{f}(y)) \mu(x-y) \, dx \, dy.
\end{equation*}
\end{lem}

\begin{remark}\label{rem:pushforward}
\begin{enumerate}
\item The Borel measure $\mu(x-y) \, dx \, dy$ is defined as the pushforward measure 
$$
\mu(x-y) \, dx \, dy:=T_\# \left(\mu(z) \, dx \, dz\right)
$$associated to
$
T:(x,z) \in \R^{2d} \mapsto (x,x+z) \in \R^{2d}.
$
We can thus change variables by
\begin{equation}
\label{change-variable}
\int_{\R^{2d}} f(x,y) \mu(x-y) \, dx \, dy =  \int_{\R^{2d}} f(x,x+z) \mu(z) \, dx \, dz
\end{equation}
for any Borel measurable $f:\R^{2d} \to [0,\infty]$.
\item Note that $\mu(x-y) \, dx \, dy$ is $\sigma$-finite because $\mu(z) \, dz$ is $\sigma$-finite. The Fubini's theorem then applies to define
$
\mu(x-y) \, dx \, dy \, d\xi:=(\mu(x-y) \, dx \, dy) \otimes d\xi,
$
etc.
\end{enumerate}
\end{remark}



Here is a version of the previous result with time-dependent functions $f$,
 as we will have to deal with such functions in the sequel.

\begin{lem}\label{cor:bf}
Assume \eqref{condition-on-mu}--\eqref{condition-symmetry} and let $f=f(t,x,\xi)$ and $\tilde{f}=\tilde{f}(t,x,\xi)$ be such that
$$
f,\nabla_x f,\nabla^2_x f \in C([0,\infty);L^1(\R^{d+1})) \quad \mbox{and} \quad
\tilde{f},\nabla_x \tilde{f} \in L^\infty(\R^+ \times \R^{d+1}).
$$Then $g_x[f] \in C([0,\infty);L^1(\R^{d+1}))$ and for almost any $t \geq 0$,
\begin{multline}
\label{eq:IBPformulawithtime} \int_{\R^{d+1}} \tilde{f} g_x[f] dx \, d\xi\\
 =\frac{1}{2} \int_{\R^{2d+1}} (f(t,x,\xi)-f(t,y,\xi)) (\tilde{f}(t,x,\xi)-\tilde{f}(t,y,\xi))  \mu(x-y) \, dx \, dy \, d\xi.
\end{multline}
\end{lem}



\subsection{Main properties of the nonlocal dissipation measure $n$
}

Let us end up with lemmas in relation with the  function defined in \eqref{def-nonloc-dissip}.  The first one is the rigorous justification of \eqref{key-taylor}.

\begin{lem}
Under the assumption \eqref{condition-on-A}, the formula \eqref{key-taylor} holds for all $S \in C^2(\R)$ and $a,b \in \R$ (with $\beta'=S' A'$).
\end{lem}

\begin{proof}
Setting $I:=\beta(b)-\beta(a)-S'(a) (A(b)-A(a))$, we have to prove that
\begin{equation}\label{main-claim}
I=\int_{\R} S''(\xi) |A(b)-A(\xi)| \Char_{\conv \{a,b\}}(\xi) \, d\xi.
\end{equation}
 This  
relies upon the following version of the Taylor's formula:
\begin{equation*}
\begin{split}
\beta(b)-\beta(a) & =\beta'(a) (b-a)+\int_a^b \beta''(\xi) (b-\xi) \, d\xi\\
& = S'(a) A'(a) (b-a)+\int_a^b \left(S''(\xi) A'(\xi)+S'(\xi) A''(\xi)\right) (b-\xi) \, d\xi.
\end{split}
\end{equation*}
We also have
$
A'(a) (b-a)=A(b)-A(a)-\int_a^b A''(\xi) (b-\xi) \, d\xi;
$
hence,
$$
I=\int_a^b S''(\xi) A'(\xi) (b-\xi) \, d\xi+\int_a^b (S'(\xi) - S'(a)) A''(\xi) (b-\xi) \, d\xi=:I_1+I_2.
$$Let us use again Taylor to rewrite the first term:
\begin{equation*}
\begin{split}
I_1 & =\int_a^b S''(\xi) \left(A(b)-A(\xi)-\int_\xi^b A''(\zeta) (b-\zeta) \, d\zeta \right) d\xi\\
 & = \int_a^b S''(\xi) (A(b)-A(\xi)) \, d\xi-\int_a^b  \int_\xi^b S''(\xi) A''(\zeta) (b-\zeta) \, d\zeta \, d\xi\\
& =:J_1-J_2.
\end{split}
\end{equation*}
By the monotonicity of $A(\cdot)$, we recognize that $J_1$ is the right-hand side of \eqref{main-claim}. It thus only remains to prove that
$
J_2=I_2.
$
But, we can rewrite $I_2$ as
$$
I_2=\int_a^b \int_a^\xi S''(\zeta) A''(\xi) (b-\xi) \, d\zeta \, d\xi
$$
and the fact that $I_2=J_2$ follows by the Fubini theorem.
\end{proof}

%

The result below will be crucial in the 
 proof of the uniqueness.

\begin{lem}\label{Lemma-F(a,b,c,d)-leq_G(a,b,c,d)}
Assume \eqref{condition-on-A}. For $a,b,c,d\in \R$, define
\begin{eqnarray*}
F(a,b,c,d) & := & \int_{\R} A'(\xi) \left(\chi(\xi;a)-\chi(\xi;b)\right) \left(\chi(\xi;c)-\chi(\xi;d)\right) d\xi,\\
G(a,b,c,d) & := & |A(b)-A(c)|\Char_{\conv\{a,b\}}(c)+|A(d)-A(a)|\Char_{\conv\{c, d\}}(a),
\end{eqnarray*}
having in mind 
\eqref{def-char}. Then
$$
\forall a,b,c,d\in\R\;\text{there holds}\; F(a,b,c,d)\leq G(a,b,c,d).
$$
\end{lem}

\begin{proof}
Note first that $\chi(\xi;a)-\chi(\xi;b)=\sgn(a-b)\Char_{\conv\{a,b\}}(\xi)$  if $\xi$ is not an extremity of $\conv\{a,b\}$,  so that
\begin{equation}\label{formula-F}
F(a,b,c,d) = \int_{\R} A'(\xi) \sgn(a-b) \sgn(c-d)\mathbf{1}_{\conv\{a,b\} \cap \conv \{c,d\}}(\xi) \, d\xi.
\end{equation}
Let us now argue in several cases according as the way $a$, $b$, $c$ and $d$ are ordered. Note first that $F$ could be nonpositive, whereas $G$ is always nonnegative. We thus do not need to consider
the cases where $F\leq 0$, which, by \eqref{formula-F}, reduces our study to
\begin{enumerate}
\item[$\bullet$] either $a \leq b$ and $c \leq d$,
\item[$\bullet$] or $b \leq a$ and $d \leq c$.
\end{enumerate}
Note next the symmetry $F(a,b,c,d)=F(c,d,a,b)$ and the analogous symmetry for $G$.
We can thus 
also 
 assume without loss of generality that $a \leq c$ in every cases, that is to say: 
\begin{enumerate}
\item[$\bullet$] either $a \leq b$, $c \leq d$ and $a \leq c$,
\item[$\bullet$] or $b \leq a$, $d \leq c$ and 
 $a \leq c$. 
\end{enumerate}
Moreover, we can assume that $\conv\{a,b\} \cap \conv \{c,d\}$ is neither empty nor reduced to a singleton, because $F$ would equal zero by \eqref{formula-F} otherwise. This allows to precise again the preceding cases by
\begin{enumerate}
\item[$\bullet$] either $a \leq c < b$ and $c < d$,
\item[$\bullet$] or $b < a$ and $d < a \leq c$.
\end{enumerate}
Let us finally divide these cases into the four following ones:
\begin{enumerate}[{\rm 1.}]
\item  either $a < c < b$ and $c < d$,
\item  or $a = c < b$ and $c < d$,
\item  or $b < a$ and $d < a < c$.
\item  or $b < a$ and $d < a=c$.
\end{enumerate}
In both the first and the second cases, we have
$$
F(a,b,c,d)=\int_{\R} \mathbf{1}_{(a,b) \cap (c,d)=(c,\min \{b,d\})}(\xi) A'(\xi) \, d\xi=A(\min \{b,d\})-A(c).
$$As far as $G$ is concerned, we have
\begin{equation*}
G(a,b,c,d)=
\begin{cases}
A(b)-A(c) & \mbox{ in the first case},\\
\frac{1}{2} (A(b)-A(c=a)) + \frac{1}{2} (A(d)-A(a=c)) & \mbox{ in the second one,}
\end{cases}
\end{equation*}
 taking into account  the monotonicity of $A(\cdot)$ and  the specific definition of $\Char$ in \eqref{def-char}. In both cases, the monotonicity of $A(\cdot)$ implies that $F(a,b,c,d) \leq G(a,b,c,d)$. We argue similarly for the third and fourth cases, which completes the proof.
\end{proof}

For the accurate proof of the uniqueness, we will need to consider truncations of the preceding quadruplet, namely $T_R(a)$, $T_R(b)$, $T_R(c)$ and $T_R(d)$, where
\begin{equation}\label{def-truncature}
T_R(u) :=
\begin{cases}
u & \mbox{if } -R\leq u \leq R,\\
\pm R & \mbox{if } \pm u>R,
\end{cases}
\end{equation}
for any given $R>0$. Here is the precise result that we will use.
\begin{lem}
Assume \eqref{condition-on-A} and $R>0$. For any reals $a$, $b$ and $\xi$, we have
\begin{equation}\label{monotonicity-bis}
\mathbf{1}_{(-R,R)}(\xi) \left(\chi(\xi;a)-\chi(\xi;b)\right)=\chi(\xi;T_R(a))-\chi(\xi;T_R(b))
\end{equation}
and
\begin{multline}\label{monotonicity}
|A(b)-A(T_R(\xi))| \Char_{\conv\{a,b\}}(T_R(\xi))\\ \geq |A(T_R(b))-A(T_R(\xi))| \Char_{\conv\{T_R(a),T_R(b)\}}(T_R(\xi))
\end{multline}
(with the representation \eqref{def-char}).
\end{lem}

\begin{proof}
The identity \eqref{monotonicity-bis} is immediate since $\mathbf{1}_{(-R,R)}(\xi) \chi(\xi;u)=\chi(\xi;T_R(u))$ for any reals $u$ and $\xi$. Let us now prove \eqref{monotonicity}. To do so, note that for any reals $b$ and $\xi$,
\begin{equation}\label{claim-monotonicity}
|A(b)-A(T_R(\xi))| \geq |A(T_R(b))-A(T_R(\xi))|;
\end{equation}
indeed, the monotonicity of $A(\cdot)$ implies that
\begin{equation*}
\begin{split}
& |A(T_R(b))-A(T_R(\xi))|\\
& =
\begin{cases}
A(R)-A(T_R(\xi))\leq A(b)-A(T_R(\xi)) & \mbox{if } b>R,\\
|A(b)-A(T_R(\xi))|& \mbox{if } -R \leq b \leq R,\\
A(T_R(\xi))-A(-R)\leq A(T_R(\xi))-A(b) & \mbox{if } b<-R,
\end{cases}
\end{split}
\end{equation*}
so that \eqref{claim-monotonicity} always holds. With \eqref{claim-monotonicity} in hands, the proof of \eqref{monotonicity} is obvious in the case where
$
\conv \{T_R(a),T_R(b) \} \subseteq \conv \{a,b\};
$
indeed, we then have
$$
\Char_{\conv\{a,b\}}(T_R(\xi)) \geq \Char_{\conv\{T_R(a),T_R(b)\}}(T_R(\xi))
$$(even if $T_R(\xi)=a$ or $b$, in which case both these functions take the value $\frac{1}{2}$).
When that inclusion fails, $a$ and $b$ are necessarily either both greater than $R$ or both lower than $-R$. The right-hand side of \eqref{monotonicity} thus equals zero everywhere, which completes the proof.
\end{proof}

\section{Equivalence between entropy and kinetic solutions}\label{sec:equiv}

This section is devoted to the proof of Theorem \ref{thm:equivalence}. Let us first justify the reformulation of the notion of entropy solutions 
 in terms of the nonlocal dissipation measure~\eqref{nei-fdpe}.

\begin{proof}[Proof of Theorem \ref{key-theo}]
Let us assume that $u \in L^\infty(\R^+;L^1(\R^d)) \cap L^\infty(\R^+ \cap \R^d)$ is an entropy solution in the sense of Definition \ref{def:entropy}, thus satisfying the inequalities \eqref{rei-fdpe}. Let $I_r$ denote the nonlocal term in $|z| > r$ of \eqref{rei-fdpe}. Applying the identity \eqref{key-taylor}, we have
\begin{equation*}
\begin{split}
I_r & =\int_{0}^{\infty} \int_{\R^d} \int_{|z| > r}  (\beta(u(t,x+z))-\beta(u(t,x))) \varphi(t,x) \mu(z)  \, dt \, dx \, dz\\
& \quad -\int_{0}^{\infty} \int_{\R^d} \int_{|z| > r} \int_{\R} S''(\xi) |A(u(t,x+z))-A(\xi)| \Char_{\conv \{u(t,x),u(t,x+z)\}}(\xi) \\
& \quad \cdot  \varphi(t,x) \mu(z)  \, dt \, dx \, dz \, d\xi\\
& =: J_r-K_r.
\end{split}
\end{equation*}
We can integrate by parts, as in the proof of Lemma \ref{lem:well-def-g}  in Appendix \ref{appendix:tech},  to rewrite
\begin{equation}\label{initial-form}
J_r=\int_{0}^{\infty} \int_{\R^d} \int_{|z| > r}  \beta(u(t,x))) (\varphi(t,x+z)-\varphi(t,x)) \mu(z)  \, dt \, dx \, dz.
\end{equation}
We then obtain \eqref{rnei-fdpe} by passing to the limit in \eqref{rei-fdpe} as $r \da 0$, thanks  to the monotone convergence theorem giving us that 
\begin{multline*}
{\lim_{r \da 0} K_r}
= \int_{0}^{\infty} \int_{\R^{d}} \int_{\R} S''(\xi)  \varphi(t,x)\\
\cdot  \underbrace{\int_{\R^{d}} |A(u(t,x+z))-A(\xi)| \Char_{\conv \{u(t,x),u(t,x+z)\}}(\xi) \mu(z)  \, dz}_{=n(t,x,\xi)} \, dt \, dx \, d\xi.
\end{multline*}

Conversely, if we now assume that \eqref{rnei-fdpe} holds, then we cut all the nonlocal terms according as $|z| >r$ or not. Putting all the $|z|>r$ parts at the left-hand side, we get that
\begin{equation*}
\begin{split}
& \mbox{\it first order $+$ initial terms} \\
& + \int_{0}^{\infty} \int_{\R^d} \int_{|z| > r}  (\beta(u(t,x+z))-\beta(u(t,x))) \varphi(t,x) \mu(z)  \, dt \, dx \, dz \\
& - \int_{0}^{\infty} \int_{\R^d} \int_{|z| > r} \int_{\R} S''(\xi) |A(u(t,x+z))-A(\xi)| \Char_{\conv \{u(t,x),u(t,x+z)\}}(\xi) \\
& \quad \cdot  \varphi(t,x) \mu(z)  \, dt \, dx \, dz \, d\xi\\
& +P.V. \int_{0}^{\infty} \int_{\R^d} \int_{|z| \leq r}  \beta(u(t,x))) (\varphi(t,x+z)-\varphi(t,x)) \mu(z)  \, dt \, dx \, dz\\
& \geq \int_{0}^{\infty} \int_{\R^d} \int_{|z| \leq r} \int_{\R} S''(\xi) |A(u(t,x+z))-A(\xi)| \Char_{\conv \{u(t,x),u(t,x+z)\}}(\xi) \\
& \quad \cdot  \varphi(t,x) \mu(z)  \, dt \, dx \, dz \, d\xi,
\end{split}
\end{equation*}
where we have done the reverse integration by parts than in \eqref{initial-form} to rewrite $J_r$ in its initial form.
The left-hand side is thus the same than in \eqref{rei-fdpe}, again by \eqref{key-taylor}. Since moreover the right-hand side is nonnegative (the test $\varphi$ being nonnegative in our considerations), we already have \eqref{rei-fdpe} and the proof is complete.
\end{proof}

With Theorem \ref{key-theo} at hand, we can establish the equivalence between entropy and kinetic solutions by following standard arguments from \cite{LiPeTa91,LiPeTa94,ChPe03,Per09}. 
Let us give details for completeness. 
 We  will use Proposition \ref{continuity-in-time}, whose proof is postponed to Appendix \ref{appendix:prop}. We will also use the two following lemmas.

\begin{lem}\label{lem-support}
Let $u_0 \in L^1(\R^d)$, assume that \eqref{condition-on-F}--\eqref{condition-symmetry} hold and suppose that the function  $u \in L^\infty(\R;L^1(\R^d)) \cap L^\infty(\R^+ \times \R^d)$
is a kinetic solution of \eqref{cauchy-problem}. The associated measures then satisfy ${\rm supp} (m+n) \subseteq \left\{\essinf u \leq \xi \leq \esssup u\right\}$.
\end{lem}
\begin{proof}
   If   $\xi \notin [\essinf u,\esssup u]$, $\chi(\xi;u)=0$  by \eqref{def-chi}
and thus
$
 \partial_\xi (m+n) =0
$
in $\mathcal{D}'((0,\infty) \times \R^d \times (\R_\xi \setminus [\essinf u,\esssup u]))$ by the first line of \eqref{kds}.
The result follows from \eqref{cond-infinity} and the last line of \eqref{kds}.
\end{proof}

\begin{lem}\label{lem-absolutely-continuous}
Let \eqref{condition-on-F}--\eqref{condition-symmetry} hold and consider
an
entropy solution $u$
of \eqref{cauchy-problem}  with initial data $u_0 \in L^1 \cap L^\infty(\R^d)$.  Then the associated nonlocal dissipation measure $n$ belongs to the space $L^1(\R^+ \times \R^{d+1})$.
\end{lem}
\noindent{Note that} this is the rigorous justification of Remark \ref{absolutely-continuous}.

\begin{proof}
 Consider $S(u)=\frac{u^2}{2}$ which is integrable in $x$. Take associated nonlinearities vanishing at zero so that $\eta(u)$ and $\beta(u)$ are also integrable in $x$. Consider the test function $\varphi_k(t) \phi(x/M)$ in \eqref{rnei-fdpe} with $0 \leq \varphi_k \in \mathcal{D}([0,\infty))$ pointwise converging to $1$ as $k \to \infty$, such that $\varphi_k' \leq 0$, and with $0 \leq \phi \in \mathcal{D}(\R^d)$ such that $\phi(0)=1$. The limit $M \to \infty$ implies that
\begin{equation*}
\begin{split}
\int_{\R^d}
S(u_0(x)) \varphi_k(0) \, dx 
\geq \int_{0}^{\infty} \int_{\R^{d+1}} S''(\xi) n(t,x,\xi) \varphi_k(t) \, dt \, dx \, d\xi,
\end{split}
\end{equation*}
thanks to the fact that $S(u) \varphi_k'=u^2 \varphi_k'/2 \leq 0$, to Lemma \ref{Lemma-for-g_x[psi]-bis} in appendix and to Fatou's lemma. Since $S''\equiv 1$, the limit $k \to \infty$ completes the proof.

\end{proof}

We will finally need a classical density result recalled below.

\begin{lem}\label{lem-density}
Let $\varphi \in \mathcal{D}((0,\infty) \times \R^{d+1})$ be nonnegative. Then it can be approximated (for the topology of $\mathcal{D}$) by functions of the form
$$
(t,x,\xi) \mapsto \sum_{i=1}^{N} \varphi_i(t,x) \phi_i(\xi),
$$for some integer $N$ and nonnegative $\varphi_i \in \mathcal{D}((0,\infty) \times \R^d)$ and $\phi_i \in \mathcal{D}(\R_\xi)$.
\end{lem}
\noindent
The property can be obtained by mollifying  $\varphi=\varphi(t,x,\xi)$  with an approximate unit of the form  $\rho_\epsilon(t,x) \theta_\epsilon(\xi)$   and discretizing the convolutions.

\begin{proof}[Proof of Theorem \ref{thm:equivalence}] Assume first that $u \in L^\infty(\R^+;L^1(\R^{d})) \cap L^\infty(\R^+ \times \R^d)$ is a kinetic solution and let us show that it is an entropy solution. 
 Recall that $\chi(\xi;u)=0$   if   $\xi \notin [\essinf u,\esssup u]$ and note that $\chi(\xi;u_0)=0$ as well by the middle line of \eqref{kds}. By Lemma \ref{lem-support}, we can then choose test functions in \eqref{rkf} of the form
$$
(t,x,\xi) \mapsto \varphi(t,x) S'(\xi), 
$$with  $\varphi \in \mathcal{D}(\R_t \times \R^d)$ and $S \in C^\infty(\R_\xi)$ convex, up to modifying $S$ for large $|\xi|$.  Using in addition the identity \eqref{fundamental-theo}, we deduce that
\begin{equation*}
\begin{split}
& \int_{0}^{\infty} \int_{\R^{d}} \left(S(u) \partial_t\varphi+\eta(u) \cdot \nabla \varphi-\beta(u) g[\varphi]\right) dt \, dx \\
& - \int_{0}^{\infty} \int_{\R^{d}} \left(S(0) \partial_t\varphi+\eta(0) \cdot \nabla \varphi-\beta(0) g[\varphi]\right) dt \, dx\\
& +\int_{\R^{d+1}}
S(u_0(x)) \varphi(0,x) \, dx-\int_{\R^d} S(0) \varphi(0,x) \, dx\\
& = \int_{0}^{\infty} \int_{\R^{d+1}}  S''(\xi) (m+n)(t,x,\xi) \varphi(t,x) \, dt \, dx \, d\xi.
\end{split}
\end{equation*}
Noticing that the sum of the second and  fourth  integrals of the left-hand side is zero, we  obtain  \eqref{rnei-fdpe} for smooth entropies. But, it is clear that $S$ can be chosen merely $C^2$ by an approximation procedure.

Conversely, assume that $u$ is an entropy solution and let us show that it is a kinetic one. We use the 
reformulation \eqref{kds} of Proposition \ref{continuity-in-time}. Since we have $u \in C([0,\infty);L^1(\R^d))$ with $u(0,\cdot)=u_0(\cdot)$ by Theorem \ref{thm:wp-entropy}, we already know that
$$
\lim_{t \da 0} \|u(t,\cdot)-u_0(\cdot)\|_{L^1(\R^d)}=0.
$$Let us now construct  $m$ as the distribution 
\begin{multline}\label{def-m}
m(t,x,\xi):=\\
\int_{{-\infty}}^{\xi} \left(\partial_t \chi(\zeta;u) + F'(\zeta) \cdot \nabla_x \chi(\zeta;u)+A'(\zeta) g_x[\chi(\zeta;u)] \right) d\zeta-n(t,x,\xi).
\end{multline}
 The integral is well defined in $\mathcal{D}'((0,\infty) \times \R^{d+1})$ since $\chi(\xi;u)=0$ if $\xi<\essinf u$. The distribution $m$ is thus well defined since $n \in L^1(\R^+ \times \R^{d+1})$ by Lemma \ref{lem-absolutely-continuous}. It satisfies the equation in \eqref{kds} by construction  and it only remains to show the other conditions of Proposition \ref{continuity-in-time}. Let us first show that
%
\begin{equation}\label{claim-support}
{{\rm supp} (m) \subseteq \{(t,x,\xi):\essinf u \leq \xi \leq \esssup u\}.}
\end{equation}
Firstly, it is immediate from \eqref{def-nonloc-dissip} that 
$$
{\rm supp} (n) \subseteq \{(t,x,\xi):\essinf u \leq \xi \leq \esssup u\}.
$$Secondly, 
\eqref{def-chi} implies that  
for any locally Lipschitz $S(\cdot)$, 
$$
\int_{{-\infty}}^{\xi} S'(\zeta)\chi(\zeta;u) \, d\zeta 
= 
\begin{cases} 
S(u)-S(0) & \mbox{if $\xi >\esssup u$},\\
0 & \mbox{if $\xi<\essinf u$}.
\end{cases}
$$Taking these facts into account in \eqref{def-m} with $S$ as the identity, $F$ and $A$, 
we find that 
$$
m=
\begin{cases}
\partial_t u+\nabla_x F(u)+g_x[A(u)]-\underbrace{(\nabla_x F(0)+g_x[A(0)])}_{=0} & \mbox{if $\xi >\esssup u$},\\
0 & \mbox{if $\xi<\essinf u$.}
\end{cases}
$$But, the remaining term of the right-hand side equals zero thanks to the weak formulation of \eqref{cauchy-problem}, see \cite{CiJa11}. This completes the proof of \eqref{claim-support}.

The test functions of Equation \eqref{def-m} can thus  also  be taken of the form
$$
(t,x,\xi) \mapsto \varphi(t,x) S''(\xi),
$$for any  $0 \leq \varphi \in \mathcal{D}((0,\infty) \times \R^d)$ and $S \in C^\infty(\R_\xi)$ convex, up to modifying $S$ for large $|\xi|$. This gives us that
\begin{equation}
\label{equivalence-last}
\begin{split}
& \langle m, \varphi S''  \rangle_{\mathcal{D}',\mathcal{D}} \\
& = \int_{0}^{\infty} \int_{\R^{d+1}} \chi(\xi;u) \left(S'(\xi) \partial_t\varphi+(S' F')(\xi) \cdot \nabla \varphi -(S'A')(\xi) g[\varphi]\right) dt \, dx \, d\xi \\
& \quad - \int_{0}^{\infty} \int_{\R^{d+1}} S''(\xi) n(t,x,\xi) \varphi(t,x)  \, dt \, dx \, d\xi,
\end{split}
\end{equation}
where we recognize the terms in \eqref{rnei-fdpe}  again by \eqref{fundamental-theo}.  Hence
$$
\langle m,  \varphi S'' \rangle_{\mathcal{D}',\mathcal{D}} \geq 0,
$$for such $\varphi=\varphi(t,x)$ and $S=S(\xi)$, which implies that $m$ is a nonnegative Radon measure on $(0,\infty) \times \R^{d+1}$ by the density claim of Lemma \ref{lem-density}. To conclude, we need the result  below. 

\begin{lem}\label{lem:cond-infinity}
The measure $m$ thus constructed on $(0,\infty) \times \R^{d+1}$ satisfies, for almost every $\xi \in \R$,
$$
\int_0^\infty \int_{\R^{d}} (m+n)(t,x,\xi) \, dt \, dx  \leq  \nu(\xi)
$$(in the sense of \eqref{nu-weak-sense}) where 
$$
\nu(\xi)= \|(u_0-\xi)^+ \mathbf{1}_{\xi > 0}\|_{L^1(\R^d)}+\|(u_0-\xi)^- \mathbf{1}_{\xi < 0}\|_{L^1(\R^d)}
$$is such that $\nu \in L^\infty_0(\R_\xi)$.
\end{lem}

Let us admit it for a while and complete the proof of Theorem \ref{thm:equivalence}.  Extending $m$ 
on $[0,\infty) \times \R^{d+1}$ by $m(\{t=0\}):=0$, we obtain a Radon measure such that $\int_0^\infty \int_{\R^d} \int_{-R}^R (m+n) \, dt \, dx \, d\xi<\infty$, for any $R>0$. The  last condition of \eqref{kds} follows from the 
 dominated  convergence theorem. The proof is complete.
\end{proof}

Let us now prove the preceding lemma.
\begin{proof}[Proof of Lemma \ref{lem:cond-infinity}]
%
 Recall that before admitting Lemma \ref{lem:cond-infinity}, $u$ was an entropy solution of \eqref{cauchy-problem} and $0 \leq m \in M^1_{\rm loc}((0,\infty) \times \R^{d+1})$ was such that  \eqref{claim-support} and \eqref{equivalence-last} hold. For any $t_0 > 0$, $0 \leq \varphi \in \mathcal{D}((0,\infty) \times \R^d)$ and $S \in C^\infty(\R_\xi)$ convex, we thus have 
\begin{equation*}
\begin{split}
& \int_{t_0}^{\infty} \int_{\R^{d+1}}\chi(\xi;u) \left(F'(\xi) \cdot \nabla_x \varphi-A'(\xi) g_x[\varphi]\right) dt \, dx \, d\xi \\
& +\int_{t_0}^{\infty} \int_{\R^{d}} S(u) \partial_t \varphi \, dt \, dx+\int_{\R^{d+1}}
S(u(t_0,x)) \varphi(t_0,x) \, dx\\
& = \int_{t_0}^{\infty} \int_{\R^{d+1}} S''(\xi) (m+n)(t,x,\xi) \varphi(t,x) \, dt \, dx \, d\xi.
\end{split}
\end{equation*}
As previously, we choose $\varphi(t,x):=\varphi_k(t) \phi(x/M)$ with $\varphi_k$ nonincreasing on $[t_0,\infty)$, pointwise converging to $1$ as $k \to \infty$, $0 \leq \phi \in \mathcal{D}(\R^d)$, and $\phi(0)=1$. Then for any $S \geq 0$ with $S(0)=0$, the successive limits $M,k \to \infty$ imply that
$$
\int_{\R^{d}}
S(u(t_0,x)) \, dx \geq \int_{t_0}^{\infty} \int_{\R^{d+1}} S''(\xi) (m+n)(t,x,\xi) \, dt \, dx \, d\xi,
$$thanks again to the fact that $S(u) \varphi_k' \leq 0$, Lemma \ref{Lemma-for-g_x[psi]-bis}, and Fatou's lemma. 
Now considering any arbitrarily given $0 \leq {\psi} \in \mathcal{D}(\R)$, we can take
$$
S(\xi):=\int_\R  \left\{(\xi-\zeta)^+ \mathbf{1}_{\zeta > 0}+(\xi-\zeta)^- \mathbf{1}_{\zeta < 0}\right\}{\psi}(\zeta) \, d\zeta
$$because it is convex, zero at zero, and nonnegative. Hence
\begin{multline*}
\int_{t_0}^{\infty} \int_{\R^{d+1}}{\psi}(\xi) (m+n)(t,x,\xi) \, dt \, dx \, d\xi\\
\leq \int_{\R^{d+1}}
\left\{(u(t_0,x)-\zeta)^+ \mathbf{1}_{\zeta > 0}+(u(t_0,x)-\zeta)^- \mathbf{1}_{\zeta < 0}\right\}{\psi}(\zeta) \, dx \, d\zeta
\end{multline*}
which is the desired result for the initial time $t_0$.  We get the result as $t_0 \da 0$ by recalling that $u(t_0,\cdot) \to u_0(\cdot)$ in $L^1(\R^d)$, see Theorem \ref{thm:wp-entropy}. 
%
\end{proof}

\section{Uniqueness and $L^1$ contraction for kinetic solutions}\label{sec:uniq}

 This section and the next one are devoted to the proof of Theorem \ref{thm:wp-kinetic}. Here we focus on the $L^1$ contraction principle that we restate below    for the reader's convenience.

\begin{theo}\label{uniqueness-theorem}
Assume \eqref{condition-on-F}--\eqref{condition-symmetry} and let $u$ and $\tilde{u}$ be two kinetic solutions of ~\eqref{cauchy-problem} with respective initial data $u_0$ and $\tilde{u}_0$ belonging to $L^1(\R^d)$. Then,
\begin{equation}\label{contraction-principle}
\|u(t,\cdot)-\tilde{u}(t,\cdot)\|_{L^1(\R^d)}\leq \|u_0-\tilde{u}_0\|_{L^1(\R^d)} \quad \mbox{for  a.e.  $t \geq 0$}.
\end{equation}
\end{theo}
%

As usually in the kinetic setting,  we first  give a formal proof 
 which  will be made rigorous later by a regularization procedure. 
 
We will follow the guidelines of \cite{ChPe03} without needing to regularize in $\xi$  similarly to what is done in   \cite{Per09}. This simplification will be also possible in our setting roughly speaking because the nonlocal dissipation measure is absolutely continuous. 

\subsection{A formal proof of uniqueness.} 
 During this formal proof of \eqref{contraction-principle},  $u(t,x)$  is  often shortly denoted by $u(x)$; this means that we  abusively  omit the time variable if there is no confusion. Moreover, $\chi(\xi;u(t,x))$  is  shortened to $\chi(\xi;u)$ (as we did many times already). Let now $m$, $n$, $\tilde{m}$ and $\tilde{n}$ be the measures associated to $u$ and $\tilde{u}$, respectively. Let us recall that $n$ and $\tilde{n}$ are given by \eqref{def-nonloc-dissip}, that is to say:
\begin{equation}
\label{measure-associated-to-two-solution-u-and-v}
\begin{split}
n(t,x,\xi) & =
 \int_{\R^d}|A(u(x+z))-A(\xi)|\Char_{\conv\{u(x),u(x+z)\}}
(\xi)\mu(z)\, dz,\\
\tilde{n}(t,x,\xi) & =
 \int_{\R^d}|A(\tilde{u}(x+z))-A(\xi)|\Char_{\conv\{\tilde{u}(x),\tilde{u}(x+z)\}}
(\xi)\mu(z)\, dz.
\end{split}
\end{equation}
As in \cite{ChPe03}, we introduce the following microscopic
contraction functional
\begin{equation}\label{contraction-function}
M(t,x,\xi):=\left(\chi(\xi;u(t,x))-\chi(\xi;\tilde{u}(t,x))\right)^2
\end{equation}
and consider its derivative with respect to time:
\begin{equation*}
\dot{{\mathcal
M}}(t):=\frac{d}{dt}\int_{\R^{d+1}}M(t,x,\xi)\, dx \, d\xi.
\end{equation*}
Having in mind the properties given in Lemma
\ref{properties-of-khi}, we see that on the one hand,
\begin{equation*}
\begin{split}
\dot{{\mathcal
M}}(t)&=\int_{\R^{d+1}} \partial_t  \left\{|\chi(\xi;u)|+|\chi(\xi;\tilde{u})|-2
\left[\chi(\xi;u)\chi(\xi;\tilde{u})\right]\right\} dx \, d\xi\\
&=\int_{\R^{d+1}}\sgn(\xi)\left(\partial_t \chi(\xi;u)+\partial_t \chi(\xi;\tilde{u})\right)
dx \, d\xi\\
& \quad -2\int_{\R^{d+1}}\left(\chi(\xi;u)\partial_t \chi(\xi;\tilde{u})+\chi(\xi;\tilde{u})\partial_t \chi(\xi;u)\right)
 dx \, d\xi
\end{split}
\end{equation*}
and on the other hand,
\begin{equation*}
\dot{{\mathcal
M}}(t)=\frac{d}{dt}\|u(t,\cdot)-\tilde{u}(t,\cdot)\|_{L^1(\R^{d})}.
\end{equation*}
So we reach to \eqref{contraction-principle} if we get
the property
$\dot{{\mathcal
M}}(t)\leq0.$

To do so, let us consider the equation of $\chi(\xi;u)$ written in \eqref{kds}. Multiplying it by $\sgn(\xi)$, we get
\begin{equation}
\label{Equation-in-khi(u)-multiply-by-sgn}
\partial_{t}|\chi(\xi;u)|+F'(\xi).\nabla_x|\chi(\xi;u)|+
A'(\xi) g_x[|\chi(\xi;u)|]
 =\sgn(\xi)\partial_\xi (m+n).
\end{equation}
In the same way, we have
\begin{equation}
\label{Equation-in-khi(v)-multiply-by-sgn}
\partial_{t}|\chi(\xi;\tilde{u})|+F'(\xi).\nabla_x|\chi(\xi;\tilde{u})|+
A'(\xi) g_x[|\chi(\xi;\tilde{u})|]
 =\sgn(\xi)\partial_\xi (\tilde{m}+\tilde{n}).
\end{equation}
Secondly, we multiply the equation of $\chi(\xi;u)$
by
$\chi(\xi;\tilde{u})$, and do similar computations for $\tilde{u}$, to get
\begin{multline}
\label{Equation-in-khi(u)-multiply-by-khi(v)}
\chi(\xi;\tilde{u})\partial_{t}\chi(\xi;u)+F'(\xi).\chi(\xi;\tilde{u})\nabla_x\chi(\xi;u)+
A'(\xi)\chi(\xi;\tilde{u})g_x[\chi(\xi;u)]\\
 =\chi(\xi;\tilde{u})\partial_\xi (m+n).
\end{multline}
and
\begin{multline}
\label{Equation-in-khi(v)-multiply-by-khi(u)}
\chi(\xi;u)\partial_{t}\chi(\xi;\tilde{u})+F'(\xi).\chi(\xi;u)\nabla_x\chi(\xi;\tilde{u})+
A'(\xi)\chi(\xi;u)g_x[\chi(\xi;\tilde{u})]\\
 =\chi(\xi;u)\partial_\xi (\tilde{m}+\tilde{n}).
\end{multline}
Now we add the equalities \eqref{Equation-in-khi(u)-multiply-by-sgn} and \eqref{Equation-in-khi(v)-multiply-by-sgn} from which we subtract twice the sum of those given in \eqref{Equation-in-khi(u)-multiply-by-khi(v)} and \eqref{Equation-in-khi(v)-multiply-by-khi(u)}. Then, after an integration over $\R^{d+1}$, we get
 \begin{equation}\label{tech-eq-uniq}
\begin{split}
\dot{{\mathcal
M}}(t)&= \int_{\R^{d+1}} (\sgn(\xi)-2\chi(\xi;\tilde{u})) \partial_\xi
(m+n) \, dx \, d\xi\\
& \quad +\int_{\R^{d+1}} (\sgn(\xi)-2\chi(\xi;u)) \partial_\xi(\tilde{m}+\tilde{n}) \, dx \, d\xi\\
& \quad {+} 2\int_{\R^{d+1}} A'(\xi) \left\{(\chi(\xi;\tilde{u}) g_x[\chi(\xi;u)]
+\chi(\xi;u) g_x[\chi(\xi;\tilde{u})] \right\} dx \, d\xi\\
& =:I_1(t)+I_2(t)+I_3(t).
\end{split}
\end{equation}Notice that we have 
omitted several terms because - at least formally - they are equal to zero, namely
\begin{equation*}
\int_{\R^{d+1}} F'(\xi) \cdot \nabla_x |\chi(\xi;u)| \, dx \, d\xi\;=\;0\;=\;\int_{\R^{d+1}} F'(\xi) \cdot \nabla_x |\chi(\xi;\tilde{u})| \, dx \, d\xi,
\end{equation*}
as well as
\begin{multline*}
\int_{\R^{d+1}} F'(\xi) \cdot \left\{\chi(\xi;\tilde{u}) \nabla_x \chi(\xi;u)+\chi(\xi;u) \nabla_x \chi(\xi;\tilde{u}) \right\} dx \, d\xi\\
\;=\;\int_{\R^{d+1}} F'(\xi) \cdot \nabla_x (\chi(\xi;u) \chi(\xi;\tilde{u})) \, dx \, d\xi \;=\;0
\end{multline*}
and
\begin{equation*}
\int_{\R^{d+1}} A'(\xi) g_x[|\chi(\xi;u)|] \, dx \, d\xi\;=\;0\;=\;\int_{\R^{d+1}} A'(\xi) g_x[|\chi(\xi;\tilde{u})|] \, dx \, d\xi.
\end{equation*}
 All these equalities stem from the use of the Fubini theorem and from the fact that, in a sense, the functions $\chi(\xi,u)$ and $\chi(\xi,\tilde u)$ vanish as $|\xi|\to \infty$ due to their integrability.
 To get the last equality, we have also (formally) used Lemma \ref{lem:bf}. 
 Now it remains to show that
$$
I_1(t)+I_2(t)+I_3(t) \leq 0
$$in \eqref{tech-eq-uniq}. For the first term, we use that
$$
\partial_{\xi} \sgn(\xi)=2 \delta(\xi) \quad \mbox{and} \quad \partial_\xi\chi(\xi;\tilde{u})=\delta(\xi)-\delta(\xi-\tilde{u}).
$$We use similar (formal) calculations for the second term and infer that
\begin{equation*}
\begin{split}
& I_1(t)+I_2(t)\\
& = -2\int_{\R^{d+1}} \left\{(\delta(\xi-\tilde{u}(x))(m+n)(t,x,\xi)+\delta(\xi-u(x))(\tilde{m}+\tilde{n})(t,x,\xi) \right\}
dx \, d\xi\\
& \leq -2 \int_{\R^{2d+1}} \delta(\xi-\tilde{u}(x))|A(u(x+z))-A(\xi)|\Char_{\conv\{u(x),u(x+z)\}}
(\xi) \mu(z) \, dx \, dz \, d\xi \\
& \quad -2 \int_{\R^{2d+1}}\delta(\xi-u(x))|A(\tilde{u}(x+z))-A(\xi)|\Char_{\conv\{\tilde{u}(x),
\tilde{u}(x+z)\}} (\xi) \mu(z)\, dx \, dz \, d\xi,
\end{split}
\end{equation*}
thanks to the nonnegativity of the measures $m$, $\tilde m$ and to explicit representations \eqref{measure-associated-to-two-solution-u-and-v} of the measures $n$, $\tilde n$. After the integration in $\xi$, we get
\begin{equation*}
\begin{split}
& I_1(t)+I_2(t)\\
& \leq -2 \int_{\R^{2d}} |A(u(y))-A(\tilde{u}(x))| \Char_{\tilde{u}(x) \in \conv\{u(x),u(y)\}} \mu(x-y) \, dx \, dy \\
& \quad -2 \int_{\R^{2d}} |A(\tilde{u}(y))-A(u(x))|\Char_{u(x) \in \conv\{\tilde{u}(x),
\tilde{u}(y)\}} \mu(x-y)\, dx \, dy,
\end{split}
\end{equation*}
where we have also 
(formally) changed the variables by $x+z \mapsto y$ and used the symmetry $\mu(y-x)=\mu(x-y)$ in \eqref{condition-symmetry}. We recognize the $G$-term of Lemma \ref{Lemma-F(a,b,c,d)-leq_G(a,b,c,d)} and thus infer that
\begin{equation*}
I_1(t)+I_2(t) \leq -2 \int_{\R^{2d}} G(u(x),u(y),\tilde{u}(x),\tilde{u}(y)) \mu(x-y) \, dx \, dy.
\end{equation*}
Further, we use Lemma \ref{lem:bf} to rewrite the last term $I_3(t)$ as follows:
\begin{multline*}
I_3(t) = 2 \int_{\R^{2d+1}} A'(\xi) \left\{\chi(\xi;u(x))-\chi(\xi;u(y))\right\}\\
\cdot
\left\{\chi(\xi;\tilde{u}(x))-\chi(\xi;\tilde{u}(y))\right\}
\mu(x-y) \, dx \, dy \, d\xi.
\end{multline*}
We recognize the $F$-term of Lemma \ref{Lemma-F(a,b,c,d)-leq_G(a,b,c,d)}. Hence
\begin{equation*}
I_3(t) = 2 \int_{\R^{2d}} F(u(x),u(y),\tilde{u}(x),\tilde{u}(y)) \mu(x-y) \, dx \, dy
\end{equation*}
and finally $\dot{{\mathcal M}}(t)= I_1(t)+I_2(t)+I_3(t) \leq 0$, since
$$
F(u(x),u(y),\tilde{u}(x),\tilde{u}(y))  \leq G(u(x),u(y),\tilde{u}(x),\tilde{u}(y)),
$$by Lemma~\ref{Lemma-F(a,b,c,d)-leq_G(a,b,c,d)}. This completes the formal proof of the $L^1$ contraction.

\subsection{Accurate uniqueness and $L^1$ contraction proof}
Let us now give the rigorous proof of Theorem \ref{uniqueness-theorem}. For brevity, we set
$
\chi=\chi(t,x,\xi)
  :=  
\chi(\xi;u(t,x))
$
and do similarly for $\tilde{u}$. Next, we follow the regularization approach
of \cite{Per09} thus considering  $\epsilon>0$ and some approximate unit 
\begin{equation*}
\rho_\epsilon(t,x):=\frac{1}{\epsilon}\rho_1\left(\frac{t}{\epsilon}\right)
\frac{1}{\epsilon^d} \rho_2\left(\frac{x}{\epsilon}\right),
\end{equation*} 
with kernels satisfying
\begin{equation*}
\begin{cases}
\text{$\rho_1 \in \mathcal{D}((-1,0))$, $\rho_2 \in \mathcal{D}(\R^d)$},\\
\text{$\rho_1,\rho_2 \geq 0$ and $\int_{\R} \rho_1=\int_{\R^d} \rho_2 =1$.}
\end{cases}
\end{equation*} 
Given $f \in L^1_{\rm loc}(\R^+ \times \R^{d+1})$ (or $M^1_{\rm loc}$), we denote its regularized version by 
$$
f_\epsilon(t,x,\xi):=(f \ast 
\rho_\epsilon 
)(t,x,\xi)=\int_0^\infty \int_{\R^d}  
\rho_\epsilon 
(t-s,x-\eta) f(s,\eta,\xi)\, ds \, d\eta.
$$This can also write
$$f_\epsilon(t,x,\xi)=\int_0^\infty \int_{\R^d}  
\rho_\epsilon 
(-s,-\eta) f(t+s,x+\eta,\xi)\, ds \, d\eta$$(with the pushforward measure, cf. Remark~\ref{rem:pushforward}).  Note that the symbol `$\ast$'  denotes   the convolution in $(x,t)$ without convoluting in $\xi$. 
We then define
 $$
M^{\epsilon}(t,x,\xi):=|\chi_{\epsilon}|+|\tilde{\chi}_{\epsilon}|-2\chi_{\epsilon} \tilde{\chi}_{\epsilon},
$$
\begin{equation}\label{preciseM}
\mathcal{M}^{\epsilon}(t):=\int_{\R^{d+1}}M^{\epsilon}(t,x,\xi)\,
dx \, d\xi \;=\; \mathcal{M}^{\epsilon}_1(t)+\mathcal{M}^{\epsilon}_2(t)+\mathcal{M}^{\epsilon}_3(t),
\end{equation}
where $\mathcal{M}^{\epsilon}_1$, $\mathcal{M}^{\epsilon}_2$ and $\mathcal{M}^{\epsilon}_3$ correspond to the contributions of the respective terms of $M^{\epsilon}(t,x,\xi)$ to the integral $\mathcal{M}^{\epsilon}(t)$.

We shall see that this is a regularized version of the microscopic contraction functional \eqref{contraction-function}.
Here is the main lemma that we will have to prove.

\begin{lem}\label{prop-convergence-of-M_epsilon_in-rigorous-proof}
Let the assumptions of Theorem \ref{uniqueness-theorem} hold and let $\epsilon>0$ be fixed. Then, we have $\mathcal{M}^\epsilon \in C^1([0,\infty))$ with $\dot{\mathcal{M}}^\epsilon(t) \leq 0$ for all $t \geq 0$.
\end{lem}

Let us admit this result for a while and complete the proof of Theorem \ref{uniqueness-theorem}. For that, we argue as in \cite{Per09}.
For the sake of completeness, we provide details.

\begin{proof}[Proof of Theorem \ref{uniqueness-theorem}]
Let us recall that
\begin{equation}
\label{estimate-chi}
\chi \in L^\infty(\R^+;L^1(\R^{d+1})) \cap L^\infty(\R^+ \times \R^{d+1})
\end{equation}
by Remark \ref{isometry}.
 It is
  standard that $\chi_\epsilon \to \chi$ in $L^1_{\textup{loc}}([0,\infty);L^1(\R^{d+1}))$, as $\epsilon \da 0$,
while remaining bounded in $ L^\infty(\R^+;L^1(\R^{d+1})) \cap L^\infty(\R^+ \times \R^{d+1})$. Recalling then that
$$
M(t,x,\xi)=(\chi-\tilde{\chi})^2=|\chi|+|\tilde{\chi}|-2 \chi \tilde{\chi}
$$and
$$
\mathcal{M}(t)=\int_{\R^{d+1}} M(t,x,\xi) \, dx \, d\xi=\|u(t,\cdot)-\tilde{u}(t,\cdot)\|_{L^1(\R^d)},
$$we infer that $\mathcal{M}(\cdot)$
is the limit of $\mathcal{M}^\epsilon(\cdot)$ in $L^1_{\rm loc}([0,\infty))$.
By Lemma~\ref{prop-convergence-of-M_epsilon_in-rigorous-proof}
$$
t \geq 0 \mapsto \|u(t,\cdot)-\tilde{u}(t,\cdot)\|_{L^1(\R^d)}
$$is essentially nondecreasing and we get \eqref{contraction-principle} by using that
$$
\lim_{t \da 0} \|u(t,\cdot)-\tilde{u}(t,\cdot)\|_{L^1(\R^d)} = \|u_0-v_0\|_{L^1(\R^d)}.
$$Let us recall that the latter limit is a consequence of Proposition \ref{continuity-in-time}
proved in Appendix~\ref{appendix:prop}. This completes the proof of Theorem \ref{uniqueness-theorem}.
\end{proof}

Let us now establish Lemma \ref{prop-convergence-of-M_epsilon_in-rigorous-proof}.  Before, we need some technical results. The two first ones work as in \cite{Per09}. Let us give details for completeness. 
 
\begin{lem}\label{lem:accurate}
Let $u \in L^\infty(\R^+;L^1(\R^{d}))$ and $\epsilon >0$. Then
$$
\chi_\epsilon \in C([0,\infty);L^1(\R^{d+1})) \cap L^\infty(\R^+ \times \R^{d+1})
$$and all its 
derivatives  in $(t,x)$ satisfy the same 
property.
\end{lem}

 The proof is immediate from \eqref{estimate-chi} since $\chi_\varepsilon=\chi \ast   \rho_\epsilon  
 $. 
%

\begin{corollary}
Let $u,\tilde{u} \in L^\infty(\R^+;L^1(\R^{d}))$. Then $\mathcal{M}^\epsilon \in C^1([0,\infty))$ with
\begin{equation}\label{rigorous-time-derivative-b}
\dot{\mathcal{M}}^\epsilon(t) = \int_{\R^{d+1}} \sgn(\xi) \left(\partial_t \chi_\epsilon+\partial_t \tilde{\chi}_\epsilon \right)dx \, d\xi\\
- 2 \int_{\R^{d+1}} \partial_t \left[ \chi_\epsilon \tilde{\chi}_\epsilon  \right] dx \, d\xi \quad \forall t \geq 0.
\end{equation}
\end{corollary}


\begin{proof}
 By \eqref{def-chi} for almost every $(t,x,\xi)$  there holds $\sgn(\chi(t,x,\xi))=\sgn(\xi)$; this property is inherited by $\chi_\epsilon$, for all $\epsilon>0$.
  Hence we have  $|\chi_\epsilon|=\sgn(\xi) \chi_\epsilon \in C([0,\infty);L^1(\R^{d+1}))$ with a time distribution derivative satisfying
$$
\partial_t |\chi_\epsilon|=\sgn(\xi) \partial_t \chi_\epsilon \in C([0,\infty);L^1(\R^{d+1})).
$$For any $\varphi \in \mathcal{D}((0,\infty))$ and $\phi \in \mathcal{D}(\R^{d+1})$, we thus have
\begin{equation*}
\begin{split}
\int_0^\infty \int_{\R^{d+1}} |\chi_\epsilon| \frac{d \varphi}{dt}(t) \phi(x,\xi) \, dt \, dx \, d\xi & = - \int_0^\infty \int_{\R^{d+1}} \varphi(t) \phi(x,\xi) \sgn(\xi) \partial_t \chi_\epsilon   \, dt \, dx \, d\xi.
\end{split}
\end{equation*}
Since we know that $\chi_\epsilon$ and $\partial_t \chi_\epsilon$ belong to $L^1_{\rm loc}([0,\infty);L^1(\R^{d+1}))$, we can  take $\phi \equiv 1$ and find that 
$$
\frac{d}{dt} \left(t \mapsto \int_{\R^{d+1}} |\chi_\epsilon| \, dx \, d\xi \right)=\int_{\R^{d+1}} \sgn(\xi) \partial_t \chi_\epsilon  \, dx \, d\xi \quad \mbox{in} \quad \mathcal{D}'((0,\infty)).
$$This gives the contribution to \eqref{rigorous-time-derivative-b} of the term $\mathcal{M}^\epsilon_1=\int_{\R^{d+1}} |\chi_\epsilon| \, dx \, d\xi$  from \eqref{preciseM}.  We argue in the same way for the terms $\mathcal{M}^\epsilon_2$, $\mathcal{M}^\epsilon_3$ and justify \eqref{rigorous-time-derivative-b} in the sense of distributions.  In particular $\mathcal{M}^\epsilon \in C^1([0,\infty))$ since the right-hand side of \eqref{rigorous-time-derivative-b}  is continuous by Lemma \ref{lem:accurate}. 
\end{proof}


 The next lemma is specific to nonlocal diffusions especially \eqref{pointwise-dissipation}. It will allow us  to avoid regularization   in $\xi$ during the whole proof of uniqueness, as mentioned previously. 

\begin{lem}\label{lem:pointwise-control}
 Assume \eqref{condition-on-F}--\eqref{condition-symmetry}  and $u$ is a kinetic solution of    \eqref{cauchy-problem} (for some $L^1$ initial data).  
Let $m$ and $n$ be the associated dissipation measures. Let $\epsilon>0$ be fixed.
 Then: 
\begin{enumerate}[{\rm (i)}]
\item \label{newii} $m_\epsilon+n_\epsilon \in W^{1,\infty}_{\textup{loc}}([0,\infty) \times \R^{d+1})$ and $\partial_\xi(m_\epsilon+n_\epsilon) \in C([0,\infty); L^1(\R^d\times K)$, for any compact $K \subset \R_\xi$,
\item \label{cond-infinity-eps} there exists
${ \nu_{\epsilon} } \in C_0(\R_\xi)$ such that
\begin{equation*}
\int_{\R^d} (m_\epsilon+n_\epsilon)(t,x,\xi) \, dx \leq { \nu_{\epsilon} }(\xi),
\end{equation*}
for any   $t \geq 0$   and $\xi \in \R$, 
\item \label{pointwise-dissipation} and for any $(t,x,\xi) \in [0,\infty) \times \R^{d+1}$,
\begin{multline*}
(m_\epsilon+n_\epsilon)(t,x,\xi)
\geq \int_{0}^\infty \int_{\R^{2d}} |A(u(t+s,x+z+\eta))-A(\xi)|\\
\cdot \Char_{\conv\{u(t+s,x+\eta),u(t+s,x+z+\eta)\}}(\xi) 
\rho_\epsilon 
(-s,-\eta) \mu(z) \, ds \, d\eta \, dz
\end{multline*}
with the 
function $\Char$ everywhere defined in \eqref{def-char}.
\end{enumerate}  
\end{lem}

\begin{proof}
First, let us prove that  $m_\epsilon$ and $n_\epsilon$
and all their  derivatives in $(t,x)$ belong to $L^\infty(\R^+ \times \R^{d+1})$. We have for instance
\begin{multline*}
\langle \partial_t m_\epsilon, \varphi \rangle_{\mathcal{D}',\mathcal{D}}\\
=\int_0^\infty \int_{\R^{d+1}} m(s,\eta,\xi) \left(\int_0^\infty \int_{\R^d} \varphi(t,x,\xi) \partial_t 
\rho_\epsilon 
(t-s,x-\eta) \, dt \, dx \right) ds \, d\eta \,  d\xi,
\end{multline*}
for any $\varphi \in \mathcal{D}((0,\infty) \times \R^{d+1})$.
 By \eqref{cond-infinity} we deduce that 
$$
{{\left|\langle \partial_t m_\epsilon, \varphi \rangle_{\mathcal{D}',\mathcal{D}}\right|} \leq \|\partial_t 
\rho_\epsilon 
\|_\infty \|\nu\|_{L^\infty(\R)} \|\varphi\|_{L^1(\R^+ \times \R^{d+1})}},
$$which proves  
that $\partial_t m_\epsilon \in L^\infty(\R^+ \times \R^{d+1})$. We argue the same way for the other derivatives in $(t,x)$ of $m_\epsilon$ and $n_\epsilon$.  
 
Now  taking the convolution of   the equation satisfied by $\chi$ in \eqref{kds} gives
\begin{equation*}
 \partial_{t}\chi_{\epsilon}+F'(\xi).\nabla_x\chi_{\epsilon}+
A'(\xi) g_x[\chi_{\epsilon}]
 =\partial_\xi (m_{\epsilon}+n_{\epsilon}) \quad \mbox{in} \quad \mathcal{D}'((0,\infty) \times \R^{d+1}).
\end{equation*}
We deduce that  
$\partial_\xi (m_{\epsilon}+n_{\epsilon}) \in L^\infty_{\textup{loc}}([0,\infty) \times \R^{d+1})$, by what precedes, which proves the first part of \eqref{newii}. The second part is also immediate from the above equation and 
Lemmas \ref{lem:accurate} and \ref{cor:bf}.

Let us now prove \eqref{cond-infinity-eps}. We use again \eqref{cond-infinity} to see that 
\begin{equation*}
\begin{split}
& \int_{\R^{d+1}}  (m_{\epsilon}+n_{\epsilon})(t,x,\xi) \varphi(\xi) \, dx \, d\xi\\
& =\int_0^\infty \int_{\R^{2d+1}} (m+n)(s,\eta,\xi) \underbrace{
\rho_\epsilon 
(t-s,x-\eta)}_{=
\rho_1 
((t-s)/\epsilon) 
\rho_2 
((x-\eta)/\epsilon) \epsilon^{-d-1}} \varphi(\xi) \, ds \, dx \, d\eta \, d\xi\\
& \leq \frac{\|
\rho_1 
\|_\infty}{\epsilon}   \int_{\R} \nu(\xi) \varphi(\xi) \, d\xi,
\end{split}
\end{equation*}
for every   $t \geq 0$   and nonnegative $\varphi \in \mathcal{D}(\R_\xi)$. Setting $C_\epsilon:=\frac{\|
\rho_1 
\|_\infty}{\epsilon}$,  we infer
$$
\int_{\R^{d}}  (m_{\epsilon}+n_{\epsilon})(t,x,\xi) \, dx \leq C_\epsilon \nu(\xi),
$$for almost every $\xi$. We can replace the right-hand side by ${\nu_{\epsilon}} \in C_0(\R_\xi)$, if choosing such a ${\nu_{\epsilon}}$ satisfying ${\nu_{\epsilon}} \geq C_\epsilon \nu$. This is the case if we take for instance
\begin{equation*}
{ \nu_{\epsilon} }(\xi):= \frac{2 C_\epsilon}{|\xi|} \int_{|\xi|/2}^{|\xi|} \esssup_{|\zeta| \geq \tau} \nu(\zeta) \, d\tau
\end{equation*}
(recall that  $\nu \in L_0^\infty(\R_\xi)$). 
 The pointwise inequality in  \eqref{cond-infinity-eps} is then easily deduced from Fatou's lemma.

Let us finally prove \eqref{pointwise-dissipation}. Only at this point, we use in passing the regularization in $\xi$. We consider a kernel $
\theta 
 \in \mathcal{D}(\R_\xi)$ that we assume to be nonnegative, even, and such that $\int 
\theta 
=1$. Let us take the approximate unit
\begin{equation*}
\theta 
_\delta(\xi):=\frac{1}{\delta} 
\theta 
\left(\frac{\xi}{\delta} \right)
\end{equation*}
and define
$(m_\epsilon+n_\epsilon)_\delta:=(m_\epsilon+n_\epsilon)\ast_\xi 
\theta 
_\delta$. For each $(t,x,\xi) \in [0,\infty) \times \R^{d+1}$, we have
\begin{equation*}
\begin{split}
(m_\epsilon+n_\epsilon)_\delta(t,x,\xi)
& = \int_{0}^\infty \int_{\R^{d+1}} (m+n)(s,\eta,\zeta) 
\rho_\epsilon 
(t-s,x-\eta)  
\theta 
_\delta(\xi-\zeta) \, ds \, d\eta \, d\zeta\\
& \geq  \int_{0}^\infty \int_{\R^{d+1}}  n(t+s,x+\eta,\xi+\zeta) 
\rho_\epsilon 
(-s,-\eta)  
\theta 
_\delta(\zeta) \, ds \, d\eta \, d\zeta\\
& = \int_{0}^\infty \int_{\R^{2d}} I_\delta(t,x,\xi;s,\eta,z)   
\rho_\epsilon 
(-s,-\eta) \mu(z) \, ds \, d\eta \, dz,
\end{split}
\end{equation*}
where $I_\delta$ stands for the expression
$$
\int_{\R} \left|A(u(t+s,x+z+\eta))-A(\xi+\zeta)\right| \Char_{\conv\{u(t+s,x+\eta),u(t+s,x+z+\eta)\}}
(\xi+\zeta)
\theta 
_\delta(\zeta) \, d\zeta.
$$Let us pass to the limit as $\delta \da 0$ in order to obtain \eqref{pointwise-dissipation}. Note first that the left-hand side always converges towards $(m_\epsilon+n_\epsilon)(t,x,\xi)$ by the
 item \eqref{newii} established above.  As far as the right-hand side is concerned, we have
\begin{equation*}
\begin{split}
\lim_{\delta \da 0} I_\delta & = \left|A(u(t+s,x+z+\eta))-A(\xi)\right|  \Char_{\conv\{u(t+s,x+\eta),u(t+s,x+z+\eta)\}}(\xi)
\end{split}
\end{equation*}
for every fixed $(s,\eta,z) \in \R^+ \times \R^{2d}$,  taking into account the everywhere representation \eqref{def-char}. Indeed, this limit exists also for $\xi$ being an extremity of the interval,
in this case, the value $\frac{1}{2}$ appears at the limit because the kernel $
\theta 
$ is even. Fatou's lemma then completes the proof of \eqref{pointwise-dissipation}.
\end{proof}

We are now ready to prove Lemma \ref{prop-convergence-of-M_epsilon_in-rigorous-proof}.

\begin{proof}[Proof of Lemma \ref{prop-convergence-of-M_epsilon_in-rigorous-proof}]
 Similarly  to   the formal computations, we will show that the right-hand side of \eqref{rigorous-time-derivative-b} is nonnegative by integrating the equations in $\chi_\epsilon$ and $\tilde{\chi}_\epsilon$.  Recall that  
\begin{equation}\label{Convolution-Equation-in-khi(u)-b}
 \partial_{t}\chi_{\epsilon}+F'(\xi).\nabla_x\chi_{\epsilon}+
A'(\xi) g_x[\chi_{\epsilon}]
 =\partial_\xi (m_{\epsilon}+n_{\epsilon})
\end{equation}
(with a similar equation for $\tilde{\chi}_\epsilon$).
Since the terms in $F'(\xi)$ and $A'(\xi)$ may not be integrable in $\xi$, we need to truncate. This amounts to rewrite \eqref{rigorous-time-derivative-b} 
as 
\begin{multline}\label{rigorous-time-derivative}
\dot{\mathcal{M}}^\epsilon(t) = \lim_{R \to \infty} \Bigg\{ \int_{\R^{d}}  \int_{-R}^R \left(\sgn(\xi) \partial_{t}\chi_{\epsilon}-2\tilde{\chi}_\epsilon \partial_t \chi_\epsilon \right) dx \, d\xi\\
+\int_{\R^{d}} \int_{-R}^R \left(\sgn(\xi) \partial_{t}\tilde{\chi}_{\epsilon}-2\chi_\epsilon \partial_t \tilde{\chi}_\epsilon \right) dx \, d\xi\Bigg\}, \quad \forall t \geq 0,
\end{multline}
and estimate the terms in brackets before passing to the limit. Note that the above limit holds by Lemma \ref{lem:accurate}. In the sequel $\epsilon>0$ is fixed.

For any $R>0$, each term of \eqref{Convolution-Equation-in-khi(u)-b} belongs to $C([0,\infty);L^{1}(\R^d \times (-R,R)))$ by Lemmas \ref{lem:accurate} and \ref{cor:bf}. In particular, we can integrate \eqref{Convolution-Equation-in-khi(u)-b} in $x \in \R^d$ and $\xi \in (-R,R)$ for any fixed $t \geq 0$. 
 Proceeding   so by previously multiplying \eqref{Convolution-Equation-in-khi(u)-b} by $\sgn(\xi)$ gives 
\begin{equation}\label{first-der-accurate}
\int_{\R^d} \int_{-R}^R  \sgn(\xi) \partial_{t}\chi_{\epsilon} \, dx \, d\xi=\int_{\R^d} \int_{-R}^R \sgn(\xi)  \partial_\xi (m_{\epsilon}+n_{\epsilon}) \, dx \, d\xi.
\end{equation}
Indeed, let us make precise that the contributions to this calculation from the convection and the nonlocal diffusion terms in \eqref{Convolution-Equation-in-khi(u)-b} vanish, due to the integration by parts in $x$.
Its validity is justified, in particular, by Lemma~\ref{lem:accurate} which gives us enough regularity to apply  \eqref{eq:IBPformulawithtime}:  We get, e.g.,
\begin{equation*}
\begin{split}
\int_{\R^d} \int_{-R}^R \sgn(\xi) A'(\xi) g_x[\chi_\epsilon] \, dx \, d\xi
& =  \frac{1}{2} \int_{\R^{2d}} \int_{-R}^R \left\{\sgn(\xi) A'(\xi)- \sgn(\xi) A'(\xi)\right\} \\
& \quad \cdot \left\{\chi_\epsilon(t,x,\xi)-\chi_\epsilon(t,y,\xi)\right\} \mu(x-y) \, dx \, dy \, d\xi,
\end{split}
\end{equation*}
which indeed equals zero. Similarly, we use Lemma~\ref{lem:accurate} and the Fubini theorem to show that
$$
\int_{\R^d} \int_{-R}^R \sgn(\xi) F'(\xi) \nabla_x \chi_\epsilon \, dx \, d\xi=0,
$$see also \cite{Per09}. 
  Let us now integrate the right-hand side of \eqref{first-der-accurate} in $\xi$ first. The regularity in Lemma \ref{lem:pointwise-control}\eqref{newii} justifies that for any $t \geq 0$ and almost every $x \in \R^d$,
\begin{equation*}
\int_{-R}^R \sgn(\xi) \partial_\xi (m_{\epsilon}+n_{\epsilon})(t,x,\xi) \, d\xi=
-2(m_{\epsilon}+n_{\epsilon})(t,x,0)+\sum_{\pm} (m_{\epsilon}+n_{\epsilon})(t,x,\pm R).
\end{equation*}
Using Lemma \ref{lem:pointwise-control}\eqref{cond-infinity-eps} to bound the terms in $\pm R$ after the integration in $x$, we conclude that for any $t \geq 0$ and $R >0$, 
\begin{equation}
\label{first-der-accurate-bis}
\int_{\R^d} \int_{-R}^R  \sgn(\xi) \partial_{t}\chi_{\epsilon} \, dx \, d\xi=-2\int_{\R^d} (m_{\epsilon}+n_{\epsilon})(t,x,0) \, dx+o_R(1),
\end{equation}
where $o_R(1) \to 0$ as $R \to \infty$.  Note that $o_R(1)$ depends on $\epsilon$ but we do not need to care about it since $\epsilon$ is fixed up to the end. 

The computation in \eqref{first-der-accurate-bis} will serve us to bound the limiting right-hand side of   \eqref{rigorous-time-derivative}.   Let us leave it aside for a while and do another computation that will be needed. 
Now let us 
multiply \eqref{Convolution-Equation-in-khi(u)-b} by $\tilde{\chi}_\epsilon$ and integrate as before. We get
\begin{equation}
\label{second-der-accurate}
\begin{split}
& \int_{\R^d} \int_{-R}^R  \tilde{\chi}_{\epsilon} \partial_{t}\chi_{\epsilon} \, dx \, d\xi \\
& =\int_{\R^d} \int_{-R}^R \tilde{\chi}_{\epsilon}  \partial_\xi (m_{\epsilon}+n_{\epsilon}) \, dx \, d\xi
-\int_{\R^d} \int_{-R}^R A'(\xi) \tilde{\chi}_{\epsilon} g_x[\chi_\epsilon] \, dx \, d\xi
+\mathcal{R}(u,\tilde{u})\\
& =:I_{R}(t)
-J_{R}(t)-\mathcal{R}(u,\tilde{u}),
\end{split}
\end{equation}
where $
\mathcal{R}(u,\tilde{u}):=-\int_{\R^d} \int_{-R}^R F'(\xi) {\tilde{\chi}}_{\epsilon} \nabla_x \chi_\epsilon \, dx \, d\xi.
$
 To compute the first integral, 
we write that 
\begin{equation*}
I_{R}(t) = \int_0^\infty \int_{\R^{2d}} \Bigg( \int_{-R}^R \tilde{\chi}(t+\tau,x+\theta,\xi) \partial_\xi (m_{\epsilon}+n_{\epsilon})(t,x,\xi) \, d\xi \Bigg) 
\rho_\epsilon 
(-\tau,-\theta)  \, d\tau \, dx \, d\theta
\end{equation*}
 where  we first integrate in $\xi$.  Recalling  the definition of
$$
\tilde{\chi}(t+\tau,x+\theta,\xi)=\chi(\xi;\tilde{u}(t+\tau,x+\theta))
$$given in \eqref{def-chi}
 and having in mind \eqref{fundamental-theo},  we get that
\begin{equation}\label{detail}
\begin{split}
I_{R}(t) & =
\int_0^\infty \int_{\R^{2d}} (m_{\epsilon}+n_{\epsilon})(t,x,T_R(\tilde{u}(t+\tau,x+\theta))) 
\rho_\epsilon 
(-\tau,-\theta) \, d\tau \, dx \, d\theta \\
& \quad -\int_{\R^d} (m_{\epsilon}+n_{\epsilon})(t,x,0) \, dx
\end{split}
\end{equation}
with the truncation function $T_R(\cdot)$ defined in \eqref{def-truncature}.  
In this reasoning, the use of \eqref{fundamental-theo} is justified by Lemma \ref{lem:pointwise-control}\eqref{newii}. We can also split the right-hand side of \eqref{detail} in two integrals since the last integral is finite by \eqref{cond-infinity-eps} of the same lemma, which implies that the first one is finite as well. Applying now the item \eqref{pointwise-dissipation},  we deduce that
\begin{equation}\label{complicate-writting}
\begin{split}
I_{R}(t) & \geq I-\int_{\R^d} (m_{\epsilon}+n_{\epsilon})(t,x,0) \, dx,
\end{split}
\end{equation}
\begin{equation*}
\begin{split}
I &:= \int_0^\infty \int_0^\infty \int_{\R^{4d}} \left|A(u(t+s,x+z+\eta))-A(T_R(\tilde{u}(t+\tau,x+\theta)))\right|\\
& \quad \cdot \Char_{T_R(\tilde{u}(t+\tau,x+\theta)) \in \conv\{u(t+s,x+\eta),u(t+s,x+z+\eta)\}}  \\
& \quad \cdot
\rho_\epsilon 
(-s,-\eta) 
\rho_\epsilon 
(-\tau,-\theta) \mu(z) \, ds \, d\tau \, dx\, d\eta \, d\theta \, dz.
\end{split}
\end{equation*}
Let us rewrite $I$ with  the pushforward measure in  \eqref{change-variable}, which amounts  to  change the variables by $(x,x+z) \mapsto (x,y)$. 
We get
\begin{equation*}
\begin{split}
I & = \int_0^\infty \int_0^\infty \int_{\R^{4d}} \left|A(u(t+s,y+\eta))-A(T_R(\tilde{u}(t+\tau,x+\theta)))\right|\\
& \quad \cdot \Char_{T_R(\tilde{u}(t+\tau,x+\theta)) \in \conv\{u(t+s,x+\eta),u(t+s,y+\eta)\}}  \\
& \quad \cdot
\rho_\epsilon 
(-s,-\eta) 
\rho_\epsilon 
(-\tau,-\theta) \mu(x-y) \, ds \, d\tau \, dx \, dy \, d\eta \, d\theta\\
& = \int_0^\infty \int_0^\infty \int_{\R^{4d}} \left|A(b)-A(T_R(c))\right| \Char_{T_R(c) \in \conv\{a,b\}}  \\
& \quad \cdot
\rho_\epsilon 
(-s,-\eta) 
\rho_\epsilon 
(-\tau,-\theta) \mu(x-y) \, ds \, d\tau \, dx \, dy \, d\eta \, d\theta,
\end{split}
\end{equation*}
with the convenient notation
\begin{equation}\label{convenient-notation}
\begin{cases}
a:=u(t+s,x+\eta),\\
b:=u(t+s,y+\eta),\\
c:=\tilde{u}(t+\tau,x+\theta),\\
d:=\tilde{u}(t+\tau,y+\theta);
\end{cases}
\end{equation}
note that $d$ will appear later when doing the same computations for $\tilde{u}$. Using in addition \eqref{monotonicity}, we deduce from \eqref{complicate-writting} that
\begin{equation}\label{accurate-first-term}
\begin{split}
I_{R}(t) & \geq
\int_0^\infty \int_0^\infty \int_{\R^{4d}} \left|A(T_R(b))-A(T_R(c))\right| \Char_{T_R(c) \in \conv\{T_R(a),T_R(b)\}}  \\
& \quad \cdot
\rho_\epsilon 
(-s,-\eta) 
\rho_\epsilon 
(-\tau,-\theta) \mu(x-y) \, ds \, d\tau \, dx \, dy \, d\eta \, d\theta\\
& \quad -\int_{\R^d} (m_{\epsilon}+n_{\epsilon})(t,x,0) \, dx.
\end{split}
\end{equation}
Let us now focus on the second integral in \eqref{second-der-accurate}. 
 Let us integrate it by parts as in \eqref{eq:IBPformulawithtime} which is again justified by Lemma \ref{lem:accurate}. We get 
\begin{equation*}
\begin{split}
J_R(t) & = \frac{1}{2} \int_{\R^{2d}} \int_{-R}^R A'(\xi) \left\{\chi_\epsilon(t,x,\xi)-\chi_\epsilon(t,y,\xi)\right\} \left\{\tilde{\chi_\epsilon}(t,x,\xi)-\tilde{\chi_\epsilon}(t,y,\xi)\right\}  \\
& \quad \cdot \mu(x-y) \, dx \, dy \, d\xi.
\end{split}
\end{equation*}
After writing the formula of the convolution products, $\chi_\epsilon=\chi \ast 
\rho_\epsilon 
$ and $\tilde{\chi}_\epsilon=\tilde{\chi} \ast 
\rho_\epsilon 
$, and using the convenient notation \eqref{convenient-notation}, we obtain that
\begin{equation*}
\begin{split}
J_R(t) & = \frac{1}{2} \int_0^\infty \int_0^\infty \int_{\R^{4d}} \int_{-R}^R A'(\xi) \left\{\chi(\xi;a)-\chi(\xi;b)\right\}  \left\{\chi(\xi;c)-\chi(\xi;d)\right\}  \\
& \quad \cdot 
\rho_\epsilon 
(-s,-\eta) 
\rho_\epsilon 
(-\tau,-\theta) \mu(x-y) \, ds \, d\tau \, dx \, dy \, d\eta \, d\theta \, d\xi.
\end{split}
\end{equation*}
Applying \eqref{monotonicity-bis}, we infer that
\begin{equation*}
\begin{split}
& J_R(t) \\
& = \frac{1}{2} \int_0^\infty \int_0^\infty  \int_{\R^{4d}} \int_{\R} A'(\xi) \left\{\chi(\xi;T_R(a))-\chi(\xi;T_R(b))\right\}
\left\{\chi(\xi;T_R(c))-\chi(\xi;T_R(d))\right\} \\
 & \quad \cdot 
 \rho_\epsilon 
 (-s,-\eta) 
 \rho_\epsilon 
 (-\tau,-\theta) \mu(x-y) \, ds \, d\tau \, dx \, dy \, d\eta \, d\theta \, d\xi.
\end{split}
\end{equation*}
Now substracting twice \eqref{second-der-accurate} to \eqref{first-der-accurate-bis}, while taking into account the preceding lower bound \eqref{accurate-first-term} of $I_R(t)$, we finally deduce that
\begin{multline}\label{accurate-final}
\int_{\R^{d}}  \int_{-R}^R \left(\sgn(\xi) \partial_{t}\chi_{\epsilon}-2\tilde{\chi}_\epsilon \partial_t \chi_\epsilon \right) dx \, d\xi
\leq
{2} \mathcal{R}(u,\tilde{u})+o_R(1)\\
+ \int_0^\infty \int_0^\infty \int_{\R^{4d}}  P_k\,
\rho_\epsilon 
(-s,-\eta) 
\rho_\epsilon 
(-\tau,-\theta) \mu(x-y) \, ds \, d\tau \, dx \, dy \, d\eta \, d\theta,
\end{multline}
\begin{equation*}
\begin{split}
P_k  & := \int_{\R} A'(\xi)\left\{\chi(\xi;T_R(a))-\chi(\xi;T_R(b))\right\}
\left\{\chi(\xi;T_R(c))-\chi(\xi;T_R(d))\right\} d\xi\\
& \quad -{2} \left|A(T_R(b))-A(T_R(c))\right| \Char_{T_R(c) \in \conv\{T_R(a),T_R(b)\}}.
\end{split}
\end{equation*}
Inverting the roles of $u$ and $\tilde{u}$, we get a similar estimate of the form
\begin{multline}\label{accurate-final-bis}
\int_{\R^{d}} \int_{-R}^R \left(\sgn(\xi) \partial_{t}\tilde{\chi}_{\epsilon}-2\chi_\epsilon \partial_t \tilde{\chi}_\epsilon \right) dx \, d\xi
\leq {2} \mathcal{R}(\tilde{u},u)+o_R(1)\\
+ \int_0^\infty \int_0^\infty \int_{\R^{4d}} \tilde{P}_k\, 
\rho_\epsilon 
(-s,-\eta) 
\rho_\epsilon 
(-\tau,-\theta) \mu(x-y) \, ds \, d\tau \, dx \, dy \, d\eta \, d\theta,
\end{multline}
\begin{equation*}
\begin{split}
\tilde{P}_k &:= \int_{\R} A'(\xi)\left\{\chi(\xi;T_R(a))-\chi(\xi;T_R(b))\right\}
\left\{\chi(\xi;T_R(c))-\chi(\xi;T_R(d))\right\} d\xi\\
& \quad -{2}  \left|A(T_R(d))-A(T_R(a))\right| \Char_{T_R(a) \in \conv\{T_R(c),T_R(d)\}}.
\end{split}
\end{equation*}
We again recognize the $F$-term and $G$-term of Lemma \ref{Lemma-F(a,b,c,d)-leq_G(a,b,c,d)} if adding \eqref{accurate-final-bis} to \eqref{accurate-final}, more precisely
$$
P_k+\tilde{P}_k
={2} F(T_R(a),T_R(b),T_R(c),T_R(d))-{2} G(T_R(a),T_R(b),T_R(c),T_R(d))
$$which is nonnegative.  Injecting the sum of \eqref{accurate-final} and \eqref{accurate-final-bis} into \eqref{rigorous-time-derivative} then  implies   
that for any $t \geq 0$, 
\begin{equation*}
\dot{\mathcal{M}}^\epsilon(t) 
\leq \lim_{R\to \infty} \left\{{2} \mathcal{R}(u,\tilde{u})+{2} \mathcal{R}(\tilde{u},u)+o_R(1)\right\}
\end{equation*}
 where $o_R(1) \to 0$ as $R \to \infty$ and
$$
\mathcal{R}(u,\tilde{u})+\mathcal{R}(\tilde{u},u)=\int_{\R^d} \int_{-R}^R F'(\xi) \nabla_x \left[{\tilde{\chi}}_{\epsilon} \chi_\epsilon \right] dx \, d\xi=0,
$$thanks to  an integration by parts in $x$ justified by  Lemma \ref{lem:accurate}; see also \cite{Per09}.  We get that $\dot{\mathcal{M}}^\epsilon(t)  \leq 0$ and complete the proof.
\end{proof}

\section{Existence of kinetic solutions}

 Let us now prove the existence part in Theorem \ref{thm:wp-kinetic} that is to say the result below. The complete proof of Theorem \ref{thm:wp-kinetic} is given just after. 

\begin{theo}\label{thm:existence}
Let~$u_0 \in L^1(\R^d)$ and assume \eqref{condition-on-F}--\eqref{condition-symmetry}.
Then there
exists at least a kinetic solution~$u$ of ~\eqref{cauchy-problem} which belongs to $C([0,\infty);L^1(\R^d))$.
\end{theo}

Theorem \ref{thm:existence} can be proven from the kinetic approach without relying on entropy solutions, in the spirit of \cite{Per09}.
However, in order to shorten the paper we will use the known existence result for entropy solutions \cite{CiJa11}. 


\begin{proof}[Proof of Theorem~\ref{thm:existence}]
Let us define $u_0^k:=T_k(u_0) \in L^1 \cap L^\infty(\R^d)$, with the truncation function of \eqref{def-truncature}. We have $u_0^k \to u_0$ in $L^1(\R^d)$ as $k  \to \infty$. Let $u_k$ be the associated entropy solutions given by Theorem \ref{thm:wp-entropy}.
 By the $L^1$ contraction principle,
$$
\|u_k-u_p\|_{C([0,T];L^1(\R^d))} \leq \|u_0^k -u_0^p\|_{L^1(\R^d)},
$$for any $T \geq 0$ and integers $k,p$. This Cauchy sequence thus converges towards some function $u \in C([0,\infty);L^1(\R^d))$ in $C([0,T];L^1(\R^d))$, for any $T>0$, and almost everywhere in $\R^+ \times \R^d$ (up to some subsequence). We will show that this $u$ is the desired kinetic solution.

By  Theorem \ref{thm:equivalence},  each $u_k$ is a kinetic solution with some measure $m_k$ satisfying the estimate of Lemma \ref{lem:cond-infinity}. Since 
$$
(T_k(u_0(x))-\xi)^\pm \mathbf{1}_{\pm \xi  >    0} \leq (u_0(x)-\xi)^\pm \mathbf{1}_{\pm \xi  >   0},
$$we deduce that for all integer $k$ and almost any $\xi \in \R$,
\begin{equation}\label{cond-infinity-last}
\int_0^\infty \int_{\R^{d}} (m_k+n_k)(t,x,\xi) \, dt \, dx  \leq  \nu(\xi),
\end{equation}
with the same fixed $\nu(\xi)=\|(u_0-\xi)^+ \mathbf{1}_{\xi  >   0}\|_{L^1(\R^d)}+\|(u_0-\xi)^- \mathbf{1}_{\xi  <   0}\|_{L^1(\R^d)}$. By the weak compactness of measures, there is some $q \in M^1_{\rm loc}([0,\infty) \times {\R^{d+1}})$ such that
$$
\int_{0}^\infty \int_{\R^{d+1}} (m_k+n_k)(t,x,\xi) \varphi(t,x,\xi)  \, dt \, dx \, d\xi \to \int_{0}^\infty \int_{\R^{d+1}} q(t,x,\xi) \varphi(t,x,\xi)  \, dt \, dx \, d\xi,
$$for any $\varphi \in C_c(\R^{d+2})$ (and up to another subsequence if necessary). This is sufficient to pass to the limit in \eqref{rkf} since $\chi(\xi;u_k) \to \chi(\xi;u)$ in $C([0,T];L^1(\R^{d+1}))$ for any $T>0$. But, we get the measure $q$ instead of $m+n$ at the right-hand side. Nevertheless, we can rewrite $q$ as $m+n$, for some nonnegative measure $m$, if we can prove that $q \geq n$. Let us do so.  For any $0 \leq \varphi \in C_c(\R^{d+2})$, 
\begin{equation}\label{lower-bound-q}
{\int_0^\infty \int_{\R^{d+1}} (m_k+n_k) \varphi  \, dt \, dx \, d\xi \geq \int_0^\infty \int_{\R^{d+1}} n_k \varphi  \, dt \, dx \, d\xi = \int_{\R^d} I_k(z) \mu(z) \, dz},
\end{equation}
\begin{equation*}
{I_k(z):=
\int_0^\infty \int_{\R^{d+1}} \underbrace{|A(u_k(t,x+z))-A(\xi)| \Char_{\conv \{u_k(t,x),u_k(t,x+z)\}}(\xi)}_{=:Q_k(t,x,z,\xi)} \varphi(t,x,\xi) \, dt \, dx \, d\xi.}
\end{equation*}
 Recall that $u_k(t,x) \to u(t,x)$ almost everywhere, let us say for $(t,x) \notin N$ with $N \subset \R^+ \times \R^d$ negligible.  For any $z \in \R^d$, we thus have $u_k(t,x+z) \to u(t,x+z)$ for any $(t,x)$ not in the negligible $N-(0,z)$. Hence 
$$
\lim_{k \to \infty} Q_k(t,x,z,\xi) = \underbrace{|A(u(t,x+z))-A(\xi)| \Char_{\conv \{u(t,x),u(t,x+z)\}}(\xi)}_{=:Q(t,x,z,\xi)},
$$for any $(t,x,\xi) \in \R^+ \times \R^{d+1}$ such that $(t,x) \notin N \cup (N-(0,z))$  and $\xi \neq u(t,x)$. We recognize the complementary of the graph of $u$;
the latter has zero Lebesgue measure in $\R^+ \times \R^{d+1}$.  Fatou's lemma then implies that 
$$
\liminf_{k \to \infty} I_k(z) \geq \int_0^\infty \int_{\R^{d+1}} Q(t,x,z,\xi)\varphi(t,x,\xi)\, dt \, dx \, d\xi \quad \forall z \in \R^d.
$$Applying again Fatou's lemma to the right-hand side of \eqref{lower-bound-q}, we obtain $q \geq n$.   We thus have reached all the conditions required in Definition \ref{defi-kinetic} excepted \eqref{cond-infinity},  but the latter is immediate from \eqref{cond-infinity-last}.  
\end{proof}



\begin{proof}[Proof of Theorem \ref{thm:wp-kinetic}]
It  only  remains to check the uniqueness of $m$. Note that the equation in \eqref{kds} determines $m$ for $t>0$, since $u$ is unique. But, the last condition of \eqref{kds}  implies moreover that $m(\{t=0\})=0$. This completes the  proof. 
\end{proof}

\appendix


\section{
%
Proofs of Lemmas \ref{lem:well-def-g}, \ref{lem:bf} and \ref{cor:bf}}\label{appendix:tech}

\begin{proof}[Proof of Lemma \ref{lem:well-def-g}]
The boundedness of $g$ follows from the formula
\begin{equation}\label{fbg}
g[f](x)
= \int_{|z|>r} (f(x+z)-f(x)) \mu(z) \, dz
+\int_{|z| \leq r} \int_0^1 (1-\tau) \nabla^2 f (x+\tau z) z^2 \mu(z) \, dz \, d\tau.
\end{equation}
The integration by parts formula \eqref{eq:IBPformula} follows, as $r\downarrow 0$, from the cutting
\begin{equation*}
\begin{split}
\int_{|z|>r} (f(x+z)-f(x))\varphi(x)  \mu(z) \, dx \, dz & = \int_{|z|>r}  f(x+z) \varphi(x) \, dx \, dz\\
& \quad -\int_{|z|>r}  f(x) \varphi(x) \mu(z) \, dx \, dz
\end{split}
\end{equation*}
and the changes of variables $x+z \mapsto x$ and $-z \mapsto z$ in the first integral.
\end{proof}

\begin{proof}[Proof of Lemma \ref{lem:bf}]
We have
\begin{equation*}
\begin{split}
I_r & :=-\int_{|z|>r} (f(x+z)-f(x) ) \tilde{f}(x)  \mu(z) \, dx \, dz \\
& = \int_{|z|>r} f(x) \tilde{f}(x) \mu(z) \, dx \, dz-\int_{|z|>r} f(x+z) \tilde{f}(x) \mu(z) \, dx \, dz\\
& = \int_{|x-y|>r} (f(x)-f(y)) \tilde{f}(x) \mu(x-y) \, dx \, dy,
\end{split}
\end{equation*}
after having rewritten both the preceding integrals with the help of \eqref{change-variable}. We can exchange the roles of $x$ and $y$ thus getting also
$$
I_r=\int_{|x-y|>r} \underbrace{(f(y)-f(x)) \mu(y-x)}_{=-(f(x)-f(y)) \mu(x-y)} \tilde{f}(y) \, dx \, dy.
$$Applying each of these formulas to half of $I_r$, we get
$$
I_r=\frac{1}{2}\int_{|x-y|>r} (f(x)-f(y)) (\tilde{f}(x)-\tilde{f}(y))  \mu(x-y) \, dx \, dy.
$$Now we conclude by passing to the limit as $r \da 0$. This is justified since
\begin{equation*}
\begin{split}
& \int_{\R^{2d}} |f(x)-f(y)| |\tilde{f}(x)-\tilde{f}(y)|  \mu(x-y) \, dx \, dy\\
& = \int_{\R^{2d}} |f(x)-f(x+z)| |\tilde{f}(x)-\tilde{f}(x+z)|  \mu(z) \, dx \, dz \quad \mbox{(by \eqref{change-variable})}\\
& \leq \|f\|_{W^{1,1}(\R^d)} \|\tilde{f}\|_{W^{1,\infty}(\R^d)} \int (|z| \wedge 2)   (|z| \wedge 2) \mu(z) \, dz,\\
\end{split}
\end{equation*}
which is finite by \eqref{condition-on-mu}.
\end{proof}

\begin{proof}[Proof of Lemma~\ref{cor:bf}]
Use Lemma \ref{lem:well-def-g} to show that $g_x[f] \in C([0,\infty);L^1(\R^{d+1}))$ and Lemma \ref{lem:bf} for the formula \eqref{eq:IBPformulawithtime}.
\end{proof}

\section{Proof of Proposition \ref{continuity-in-time}
}\label{appendix:prop}


 We follow the guidelines of \cite{ChPe03,Per09}. Let us first give technical results. 

\begin{lem}\label{Lemma-for-g_x[psi]-bis}
Let $\phi,\varphi \in C^\infty_b(\R^d)$ and $\phi_M(x):=\phi(x/M)$. Then $g[\phi_M \varphi] \to \phi(0) g[\varphi]$ pointwise as $M\to \infty$ while being bounded uniformly in large $M$. In particular $g[\phi_M] \to 0$ pointwise.
\end{lem}

\begin{proof}
The uniform bound in large $M$ holds since $g:C_b^2(\R^d) \to C_b(\R^d)$ is bounded.
For the convergence, use \eqref{fbg} to write 
\begin{multline*}
\left|g[\phi_M \varphi-\phi(0)\varphi](x)\right| \\ \leq
\int_{|z|>r} \left|(\phi_M \varphi-\phi(0)\varphi)(x+z)-(\phi_M \varphi-\phi(0)\varphi)(x) \right| \mu(z) \, dz+C \int_{|z| \leq r} z^2 \mu(z) \, dz
\end{multline*}
with $C$ independent of $r>0$ and large $M$, and let then $M \to \infty$ and $r \da 0$ successively.
\end{proof}


We can then use test functions as below.

\begin{lem}
Assume \eqref{condition-on-F}--\eqref{condition-symmetry} and $u \in L^\infty(\R^+;L^1(\R^d))$ is a kinetic solution of \eqref{cauchy-problem}  with initial data $u_0 \in L^1(\R^d)$. Then, for any Lebesgue point $T$ of the locally integrable  function $t \in \R^+ \mapsto u(t) \in L^1$ and any $\varphi \in C_b^\infty(\R^d \times \R_\xi)$ compactly supported in $\xi$, 
\begin{equation}\label{keycont}
\begin{split}
& \int_{\R^{d+1}} \chi(\xi;u_0(x)) \varphi(x,\xi) \, dx  \, d\xi +\int_0^{T}\int_{\R^{d+1}} \chi(\xi;u) \left(F'(\xi)
  \cdot \nabla_x\varphi - A'(\xi)
  g_x[\varphi]\right) dt \,dx\,d\xi\\
  & =\int_{\R^{d+1}}\chi(\xi;u(T,x))\varphi(x,\xi) \, dx\, d\xi +\int_0^{T}\int_{\R^{d+1}}(m+n)(t,x,\xi)\partial_\xi
   \varphi(x,\xi)\, dt\,dx\,d\xi.
\end{split}
\end{equation}
\end{lem}

\begin{proof}
Take the test $\varphi_k(t) \phi(x/M) \varphi(x,\xi)$ in \eqref{rkf} where  $\varphi_k(t):=1
  - 
\int_0^t \rho_k$ for some standard mollifier $\rho_k=\rho_k(t)$ approximating 
  $\delta(t=T)$  
as $k \to \infty$, and $\phi \in \mathcal{D}(\R^d)$ satisfies   $\phi(0)=1$. Let then $k,M \to \infty$ successively by using Lemma \ref{Lemma-for-g_x[psi]-bis} 
  and the dominated convergence theorem.   
\end{proof}

Let us continue with standard results. The first one is the de La Vall\'ee Poussin criterion for weak compactness in $L^1$ (or equivalently weakly sequentially compactness by Eberlian  \v{S}mulian   theorem), cf. e.g. \cite[p. 19]{Mey66}. 

\begin{lem}\label{poussin}
For any ball $B \subset \R^d$, $\{u_k\}_k \subset L^1(B)$ is relatively weakly compact if and only if there exists a nonnegative and convex function $\Phi$ such that $\lim_{|\xi| \to \infty} \frac{\Phi(\xi)}{|\xi|}=\infty$ and $\{\Phi(u_k)\}_k$ is bounded in $L^1(B)$.
\end{lem}

To establish the strong convergence, we will argue on the weak-$\star$ limit of the kinetic functions and the properties below will be needed. 

\begin{lem}\label{perthame2}
Let us assume that $u_k \rightharpoonup \overline{u}$ in $L^1_{\rm loc}(\R^d)$--$w$ and $\chi(\xi;u_k(x)) \rightharpoonup \overline{\chi}(x,\xi)$ in $L^\infty(\R^{d+1})$--$w$$\star$. Then:
\begin{enumerate}[\rm (i)]
\item $\overline{\chi} \in L^1(B \times \R_\xi)$ for any ball $B \subset  \R^d$, 
\label{itemint}
\item $\overline{u}(x)=\int \overline{\chi}(x,\zeta) \, d\zeta$ and $\int_{-\infty}^\xi (\overline{\chi}(x,\zeta) - \chi(\zeta;\overline{u}(x))) \, d\zeta \leq 0$ for a.e. $(x,\xi)$, \label{prop-perth}
\item $u_k \to \overline{u}$ strongly in $L^1_{\rm loc}(\R^d)$ when $\overline{\chi}  = \chi(\xi;\overline{u})$. \label{itemstrong}
\end{enumerate}
\end{lem}
 
See \cite[Lem  2.3.1]{Per09} for the proof of \eqref{itemint}, \cite[Lem 2.3.3]{Per09} for the first part of \eqref{prop-perth} and \eqref{itemstrong}, and \cite[Thm 2.2.1]{Per09} for the second part of \eqref{prop-perth}.
%
We are now in position to prove the time continuity of kinetic solutions at $t=0$ by reproducing the arguments of \cite{ChPe03,Per09}.

\begin{proof}[Proof of Proposition \ref{continuity-in-time}]
 Let $u_0 \in L^1(\R^d)$, $u \in L^\infty(\R^+;L^1(\R^d))$ and $0 \leq m \in M_{\rm loc}^1([0,\infty) \times \R^{d+1})$ be such that \eqref{cond-infinity} holds. We have to prove that $\eqref{rkf} \Longleftrightarrow \eqref{kds}$, the difference being in the sense of the initial datum. 

\medskip

{\bf Claim $\eqref{kds} \Rightarrow \eqref{rkf}$.} This follows from a standard approximation procedure of test functions in $\mathcal{D}([0,\infty) \times \R^{d+1})$ by test functions in $\mathcal{D}((0,\infty) \times \R^{d+1})$. The details are left to the reader.

 \medskip 
  
{\bf Claim $\eqref{rkf} \Rightarrow \eqref{kds}$:} {\it Strong continuity in $L^1_{\rm loc}(\R^d)$.} The main property to establish is the second line of \eqref{kds}, on which we focus now. Consider an even, nonnegative and strictly convex $C^\infty$ function $\Phi$ such that $\lim_{|\xi| \to \infty} \frac{\Phi(\xi)}{|\xi|}=\infty$ and
$\int\Phi (u_0)<\infty$.\footnote{
Take e.g. a regular version of $\xi \mapsto \sum_{k \geq 1} (|\xi|-r_k)^+$ where $\sum_{k \geq 1} \int_{|u_0| \geq r_k} |u_0|<\infty$ for some fixed $0=r_1<r_2<\dots$}. We claim that for any Lebesgue point $T$ of $t \mapsto u(t)$,
\begin{equation}\label{Inequality-in-Lemma-int-Phi(u)+int-Phi''(xi)(m+n)-less-than}
\int_{\R^d}\Phi(u(T,x))\,
dx+\int_0^T\int_{\R^{d+1}}\Phi''(\xi)(m+n)(t,x,\xi)\,
dt\,dx\,d\xi= \int_{\R^d}\Phi(u_0)\, dx.
\end{equation}
To prove this, take $\varphi(x,\xi):= \phi(x/M) \psi(\xi/R) \Phi'(\xi)$  in \eqref{keycont} with $\phi \in \mathcal{D}(\R^d)$ and $0 \leq \psi \in \mathcal{D}(\R)$ such that $\psi$ is even, nonincreasing on $\R^+$, and $\phi(0)=\psi(0)=1$. Letting $M\to \infty$, 
\begin{multline*}
 \int_{\R^{d+1}} \chi(\xi;u_0(x)) \psi(\xi/R) \Phi'(\xi)\, dx \, d\xi=\int_{\R^{d+1}} \chi(\xi;u(T,x)) \psi(\xi/R) \Phi'(\xi)\,
dx \, d\xi\\
+\int_0^T\int_{\R^{d+1}}(m+n)(t,x,\xi) \psi(\xi/R) \Phi''(\xi)\,
dt\,dx\,d\xi\\
+\frac{1}{R} \int_0^T\int_{\R^{d+1}}(m+n)(t,x,\xi) \psi'(\xi/R) \Phi'(\xi)\,
dt\,dx\,d\xi,
\end{multline*}
thanks to Lemma \ref{Lemma-for-g_x[psi]-bis} as well as \eqref{cond-infinity} and the dominated convergence theorem. To continue, we need that to cancel the last integral as $R \to \infty$. This will be the case if $\nu \Phi' \in L^\infty_0(\R)$   with $\nu$ from \eqref{cond-infinity}. For this sake,   it suffices to fix a smaller $\Phi$ from the begining, if necessary, in order to have also this property. Letting $R \to \infty$ then implies
\eqref{Inequality-in-Lemma-int-Phi(u)+int-Phi''(xi)(m+n)-less-than} thanks to the monotone convergence theorem to handle all the other terms.

Consider now Lebesgue points $t_k$ of $t \mapsto u(t)$ such that $t_k \da 0$ as $k \to \infty$. By \eqref{Inequality-in-Lemma-int-Phi(u)+int-Phi''(xi)(m+n)-less-than} and Lemma \ref{poussin}, $\{u_k:=u(t_k,\cdot)\}_k$ is relatively weakly compact in $L^1_{\rm loc}(\R^d)$. Hence
\begin{equation*}
u_k \rightharpoonup \overline{u} \quad \text{in} \quad L^1_{\rm loc}(\R^d)\text{--}w, \quad  \chi(\xi;u_k(x)) \rightharpoonup \overline{\chi}(x,\xi) \quad \text{in} \quad L^\infty(\R^{d+1})\text{--}w\star,
\end{equation*}
up to taking a subsequence if necessary. Taking eventually another subsequence, 
\begin{equation*}
\mbox{$\displaystyle\int_0^{t_k} m(t,x,\xi) \, dt \rightharpoonup \overline{m}(x,\xi) \quad \text{in} \quad M^1_{\rm loc}(\R^{d+1})$--$w$}
\end{equation*}
thanks to \eqref{cond-infinity}. 
Letting $T=t_k\downarrow 0$ in \eqref{keycont} implies that
\begin{equation*}
 \begin{split} 
\int_{\R^{d+1}} \chi(\xi;u_0(x))\varphi(x,\xi) \, dx \, d\xi & = \int_{\R^{d+1}}\overline{\chi}(x,\xi)  \varphi(x,\xi) \, dx \,d\xi\\
 & \quad + \int_{\R^{d+1}}\overline{m}(x,\xi) \partial_\xi \varphi(x,\xi)\, dx\,d\xi,
   \end{split}
\end{equation*}
that is 
$\overline{\chi}(x,\xi)-\chi(\xi;u_0(x))=\partial_\xi
\overline{m}(x,\xi)$. Since $\int_{\R^d} \overline{m}(x,\xi) \, dx \leq \nu(\xi) \in L^\infty_0(\R)$ by stability of \eqref{cond-infinity} at the weak limit, 
$\overline{m}(x,\xi)=\int_{-\infty}^{\xi} (\overline{\chi}(x,\zeta)-\chi(\zeta;u_0(x))) \, d\zeta$
thanks to the item \eqref{itemint} of Lemma \ref{perthame2}. The limit as $\xi \to +\infty$ then implies that $\int \overline{\chi}(x,\zeta) \, d\zeta=u_0(x)$ and it follows that $\overline{u} = u_0$ and $\overline{m} \leq 0$ by \eqref{prop-perth}. 
This nonnegative measure is thus zero which implies that $\overline{\chi}(x,\xi)=\chi(\xi;u_0(x))$ and $u_k \to u_0$ strongly in $L^1_{\rm loc}(\R^d)$ by \eqref{itemstrong}. 


\medskip 
  
{\bf Claim $\eqref{rkf} \Rightarrow \eqref{kds}$:} {\it Strong continuity in $L^1(\R^d)$.} 
Let $R>0$ be arbitrarily fixed and let us prove that
\begin{equation}\label{claim-int-inf}
\sup_{k \in \mathbb{N}} \int_{|x| \geq M} |T_R(u_k)| \, dx \to 0 \quad \mbox{as} \quad M \to \infty,
\end{equation}
with $T_R$ from \eqref{def-truncature}. Consider a regularization of $\frac{d}{d\xi} |T_R(\xi)|$ given by $S_R \ast 
\theta_\delta 
$ where $S_R(0)=0$,
$$
S_R'(\xi):=
\begin{cases}
\sgn(\xi) & \mbox{for $|\xi| \leq R$,}\\
\sgn(\xi)(R+1-|\xi|) & \mbox{for $R<|\xi|<R+1$},\\
0 & \mbox{otherwise,} \\
\end{cases}
$$and $0 \leq 
\theta 
_\delta \in \mathcal{D}(\R)$ is an approximate unit as $\delta \da 0$. Note that
\begin{equation}\label{needed}
|\cdot| \geq S_R(\cdot) \geq |T_R(\cdot)| \quad \text{and} \quad (S_R \ast 
\theta_\delta 
)'' \geq-\mathbf{1}_{(-R-1,-R)\cup(-R,R+1)} \ast 
\theta_\delta 
.
\end{equation}
Now choose $\varphi(x,\xi):=\phi_M(x) (S_R \ast 
\theta_\delta 
)' (\xi)$
in \eqref{keycont} where $\phi_M(x)=\phi(x/M)$ with
$$
\phi(x):=
\begin{cases}
0 & \mbox{for } |x| \leq 1/2,\\
1 & \mbox{for } |x| \geq 1,\\
\end{cases}
\quad \mbox{and} \quad 0\leq \phi \leq 1 \quad \mbox{elsewhere.}
$$Doing this with $T=t_k$, we infer that
\begin{equation*}
\begin{split}
& \int_{|x|>\frac{M}{2}} S_R \ast 
\theta_\delta 
(u_0(x)) \, dx +C_R \int_0^{t_k}\int_{\R^{d+1}}  |\chi(\xi;u)| \left(|\nabla \phi_M|+|g[\phi_M]|\right) dt \,dx\,d\xi\\
  &  \geq \int_{|x|>M}S_R \ast 
  \theta_\delta 
  (u(t_k,x)) \, dx+\int_0^{t_k}\int_{\R^{d+1}}(m+n)(t,x,\xi) \phi_M(x)  (S_R \ast 
  \theta_\delta 
  )'' (\xi)
   \, dt\,dx\,d\xi,
\end{split}
\end{equation*}
for some Lipschitz constant $C_R$ of $F$ and $A$ on the support of $(S_R \ast 
\theta_\delta 
)'$, thus independent of small $\delta$. 
Letting $\delta \to 0$ while using \eqref{needed} then implies that
\begin{equation*}
\begin{split}
\int_{|x|>M} |T_R(u_k)| \, dx & \leq  \int_{|x|>\frac{M}{2}} |u_0| \, dx\\
& \quad +C_R\int_0^{t_k}\int_{\R^{d+1}}  |\chi(\xi;u)|  \left(|\nabla \phi_M|+|g[\phi_M]|\right) dt \,dx\,d\xi\\
& \quad + \int_0^{t_k}\int_{|x|>\frac{M}{2}} \int_{
R 
  \leq   
|\xi|
  \leq  
R+1} (m+n) \, dt\,dx\,d\xi.
\end{split}
\end{equation*}
The claim \eqref{claim-int-inf} is now obtained by using Lemma \ref{Lemma-for-g_x[psi]-bis} to cancel the penultimate integral as $M \to \infty$, as well as \eqref{cond-infinity} and the dominated convergence theorem for the last integral.

\medskip 
  
{\bf Conclusion.} By \eqref{claim-int-inf} and the previous $L^1_{\rm loc}$ convergence, $T_R(u_k) \to T_R(u_0)$ in $L^1(\R^d)$ for any $R>0$ up to a subsequence. Hence  $u_k \to u_0$ in $L^1(\R^d)$ since
$
\lim_{R \to \infty} \sup_{k} \int_{|u_k| > R} |u_k|=0
$
by \eqref{Inequality-in-Lemma-int-Phi(u)+int-Phi''(xi)(m+n)-less-than}. This completes the proof of the middle line of \eqref{kds}. Since the first line is immediate and the last one is a consequence of \eqref{Inequality-in-Lemma-int-Phi(u)+int-Phi''(xi)(m+n)-less-than}, the whole proof is complete.
\end{proof}

\bibliographystyle{plain}

\end{document}